\documentclass[12pt]{article}
\usepackage[english]{babel}
 \usepackage[pagewise]{lineno}\linenumbers
\usepackage{amscd,amsmath,amssymb,amsthm,amsfonts,epsfig,graphics}

\usepackage{graphicx}
\usepackage{geometry}
\usepackage{hyperref}  
\hypersetup{colorlinks=true, linkcolor=blue, anchorcolor=blue, 
citecolor=red, filecolor=blue, menucolor=blue,
urlcolor=blue}

\numberwithin{equation}{section}
\newtheorem{theorem}{Theorem}[section]
\newtheorem{corollary}[theorem]{Corollary}
\newtheorem{proposition}[theorem]{Proposition}
\newtheorem{lemma}[theorem]{Lemma}
\theoremstyle{remark}
\newtheorem{remark}{Remark}[section]

\theoremstyle{definition}

\newcommand{\R}{\mathbb{R}}
\newcommand{\C}{\mathbb{C}}
\newcommand{\N}{\mathbb{N}}

\newcommand{\lap}{\bigtriangleup}

\newcommand{\D}{\mathcal D}
\newcommand{\re}{\textrm{Re}}

\newcommand{\an}[1]{\langle #1 \rangle}

\newcommand{\grad}{\bigtriangledown}

\def\XXint#1#2#3{{\setbox0=\hbox{$#1{#2#3}{\int}$ }
\vcenter{\hbox{$#2#3$ }}\kern-.58\wd0}}
 
\footnotetext{\emph{Subjclass[2010]:} {35Q41, 35L05.}}

 
\begin{document}



\author{Federico Cacciafesta\footnote{Dipartimento di Matematica, Universit$\grave{\text{a}}$ degli studi di Padova,  Via Trieste, 63, 35131 Padova PD - Italy.}, Anne-Sophie de Suzzoni\footnote{Universit\'e Paris 13, Sorbonne Paris Cit\'e, LAGA, CNRS ( UMR 7539), 99, avenue Jean-Baptiste Cl\'ement, F-93430 Villetaneuse, France}}



\title
{Weak dispersion for the Dirac equation on asymptotically flat and warped products spaces.}

\maketitle
\begin{abstract}
 In this paper we prove local smoothing estimates for the Dirac equation on some non-flat manifolds; in particular, we will consider asymptotically flat and warped products metrics. The strategy of the proofs relies on the multiplier method.
\end{abstract}


%
%


\section{Introduction}
The \emph{Dirac equation on $\mathbb{R}^{1+3}$} is a
constant coefficient, hyperbolic system of the form
\begin{equation}\label{eq1}
iu_t+\mathcal{D}u+m\beta u=0\\
\end{equation}
where $u:\mathbb{R}_t\times\mathbb{R}_x^3\rightarrow\mathbb{C}^4$, $m\ge0$ is called the \emph{mass}, the \emph{Dirac operator} is defined as
\begin{equation*}
  \mathcal{D}=
  i^{-1}\displaystyle\sum_{k=1}^n\alpha_k
  \frac{\partial}{\partial x_{k}}=
  i^{-1}(\alpha\cdot\grad),
\end{equation*}
and the $4\times 4$ Dirac matrices can be written as
\begin{equation}\label{diracmatrices}
\alpha_k=\left(\begin{array}{cc}0 & \sigma_k \\\sigma_k & 0\end{array}\right),\quad k=1,2,3,\qquad\beta=
\left(\begin{array}{cc}I_2 & 0 \\0 & -I_2\end{array}\right)
\end{equation}
in terms of the Pauli matrices
\begin{equation}
\sigma_1=\left(\begin{array}{cc}0 & 1 \\1 & 0\end{array}\right),\quad
\sigma_2=\left(\begin{array}{cc}0 &
-i \\i & 0\end{array}\right),\quad
\sigma_3=\left(\begin{array}{cc}1 & 0\\0 & -1\end{array}\right).
\end{equation}
The $\alpha$ matrices satisfy the following relations
$$
 \alpha_j\alpha_k+\alpha_k\alpha_j=2\delta_{jk}I_4,\qquad
  1\le j,k\le 3,
$$
$$
\alpha_j\beta+\beta\alpha_j=0,\qquad j=1,2,3,
$$
$$
\beta^2=I_4;
$$
as a consequence, the following identity holds
\begin{equation}\label{square}
 (i\partial_t-\mathcal{D}-m\beta)(i\partial_t+\mathcal{D}+m\beta)u=
  (\Delta-m^2-\partial_{tt}^2)\mathbb{I}_4u.
\end{equation}
This identity allows us to study the free Dirac equation through a system of decoupled Klein-Gordon (or wave, in the mass-less case) equations. Therefore, it is not a difficult task to deduce dispersive estimates (time-decay, Strichartz...) for the Dirac flow from the corresponding ones of their more celebrated Klein-Gordon or wave counterparts. Of course, when perturbative terms appear in equation \eqref{eq1}, as potentials or nonlinear terms, the argument above needs to be handled with a lot of additional care, and in particular is going to fail in low regularity settings, when the structure of each term play crucial roles. The study of dispersive estimates for the Dirac equation with potentials has already been dealt with in literature: we mention at least the papers \cite{boussdanfan, Cacciafesta2, cacdan, DanconaFanelli2, DanconaFanelli08-a, cacser} in 
which various sets of estimates are discussed for electric and magnetic perturbations of equation \eqref{eq1}.

In the last years, a lot of effort has been spent in order to investigate higher order perturbations of dispersive partial differential equations: in particular, the problem of understanding how variable coefficients perturbations affect the wave and Schr\"odinger flows has attracted increasing interest in the community.  The interest for this kind of problems is of geometric nature, as it is indeed natural to interpret the variable coefficients as a "change of metrics", and therefore to recast the problem as the study of dispersive dynamics on non-flat manifolds. It turns out that in this context a crucial role is played by the so called \emph{non-trapping condition} on the coefficients that, roughly speaking, is a condition that prevents geodesic flows to be confined in compact sets for large times: the failure of such a condition is indeed understood to be an obstacle for dispersion. Such a condition is in fact guaranteed in case of "small perturbations" of the flat metric. On the subject of dispersion 
for 
Klein-Gordon and wave equations, we mention, in a non exhaustive way, \cite{kg,wave00,wave0,wave01,wave1,wave2,wave3}.

The aim of this manuscript is to provide some first results in this framework for the Dirac equation for which, to the best of our knowledge, nothing is known; in particular we here aim to prove weak dispersive estimates for its flow under some different assumptions on the geometry. We stress the fact that, due to to the rich algebraic structure of the Dirac operator, its generalization to curved spaces is significantly more delicate than the one of the Laplacian;  we dedicate section \ref{curv} to this issue. On the other hand, once the equation is correctly settled, it is possible to rely on the the squaring trick \eqref{square} as in the free case to reduce to a suitable variable coefficients wave equation with a lower order term, for which the multiplier technique can be exploited. Therefore, in the present paper we will essentially be mixing the strategy developed in \cite{boussdanfan} to prove dispersive estimates for the magnetic Dirac equation, with  \cite{cacsmooth, cacdanluc}, in 
which the same method is adapted to deal with the more involved variable coefficients setting for the Schr\"odinger and Helmholtz equations.

We will show in section \ref{curv} that the general form of the Dirac operator on a manifold with a given metric $g_{\mu\nu}$ is the following
\begin{equation}\label{eq}
\mathcal{D}=i\gamma^a e^\mu_a D_\mu
\end{equation}
where the matrices $\gamma_0=\beta$ and $\gamma^j=\gamma^0\alpha_j$ for $j=1,2,3$, $e^\mu$  is a \emph{vierbein} (i.e. a set of matrices that, roughly speaking, connect the curved spacetime to the Minkowski one) and $D_\mu$ defines the covariant derivative for fermionic fields.

In what follows, we shall restrict to metrics $g_{\mu\nu}$ having the following structure
\begin{equation}\label{struct}
g_{\mu\nu} = \left \lbrace{\begin{array}{ll}
\phi^{-2} (t) & \textrm{ if } \mu= \nu = 0\\
0 & \textrm{ if } \mu\nu = 0 \textrm{ and } \mu\neq \nu\\
-h_{\mu\nu}(\overrightarrow x) & \textrm{ otherwise. }\end{array}} \right.
\end{equation}
The function $\phi$ is assumed to be strictly positive for all $t$; let us remark that, after a change of variable on time, one may take $\phi$ equal to $1$.
The structure \eqref{struct} implies means, in simple words, that time and space are decoupled. What is more, we assume that the manifold is complete: this ensures that the Dirac operator is self-adjoint (see \cite{selfa}), a property that is crucial in order to guarantee a unitary dynamics, and that we use for the conservation of energy or for estimates of norms in terms of this operator. The same assumption is made in \cite{wave01}. Within this setting, we will show that equation $(\D+m)u=0$ can be written in the more convenient form
\begin{equation}\label{diraceqcurv}
i\phi \partial_t u - H u = 0
\end{equation}
where $H$ is an operator such that $H^2=-\Delta_h+\frac14\mathcal{R}_h+m^2$, and $\Delta_h$ and $\mathcal R_h$ are respectively the Laplace-Beltrami operator and the scalar curvature associated to the spatial metrics $h$. As a consequence, it can be proved that if $u$ solves equation \eqref{diraceqcurv} then $u$ also solves the equation
\begin{equation}\label{wavevar}
-(\phi \partial_t)^2 u  + \lap_h u - \frac14 \mathcal R_h u -m^2 u = 0.
\end{equation}
We should stress the analogy with the free case; the scalar curvature term that appears in the equation above vanishes when reducing to the Minkowski metric. Moreover, what we mean by $\lap_h$ is actually the Laplace-Beltrami for spinors, that is $\lap_h = D^jD_j $, with $D_j$ the covariant derivatives acting on spinors (see section \ref{curv}).
\medskip

Our first main result concerns the case of manifolds which are asymptotically flat; let us explain precisely the assumptions in this case. First of all, we require for $h(x)=[h_{jk}(x)]_{j,k=1}^{n}$ the following natural matrix-type bounds to hold for every $x$, $\xi\in\mathbb{R}^3$
\begin{equation}\label{eq:assaabove}
 \nu |\xi|^2\leq h^{jk}(x)\xi_j\xi_k\leq N|\xi|^2
\end{equation}
where $h_{inv}:=[h^{jk}(x)]_{j,k=1}^{n}$ is the inverse of the matrix $h(x)$, which is equivalent to
\begin{equation}\label{eq:assaabove2}
 N^{-1}  |\xi|^2\leq h_{jk}(x)\xi_j\xi_k\leq \nu^{-1} |\xi|^2.
\end{equation}
Notice that a consequence of \eqref{eq:assaabove} is that there exist constants $C_{i,\nu,N}$, $i=1,2,3$ depending on $\nu, N$, such that for all $v \in \R^3$,
$$|h_{inv}v|^2\leq C_{1,\nu,N} |hv|^2\leq C_{2,\nu,N}|v|^2\leq C_{3,\nu,N} |h_{inv}v|^2$$
Moreover, \eqref{eq:assaabove} implies
\begin{equation}\label{detest}
N^{-3/2}\leq\sqrt{\det(h(x))}\leq \nu^{-3/2},\qquad\forall\:x\in\mathbb{R}^3
\end{equation}
where $\det(h(x))=\det[h_{jk}(x)]_{j,k=1}^{n}$.
Then, we impose an asymptotically flatness condition in the form
\begin{equation}\label{assasy}
 |h_{inv}(x) -I|\leq C_I\langle x\rangle^{-\sigma},\qquad C_I<1
\end{equation}
and
\begin{equation}\label{assdera}
 |h_{inv}'(x)|+|x||h_{inv}''(x)|+|x|^{2}|h_{inv}'''(x)|\le C_{h} 
  \langle x\rangle^{-1-\sigma},\qquad \sigma\in(0,1),
\end{equation}
where we are using the standard notation for the japanese bracket $\langle x\rangle=(1+|x|^2)^{1/2}$ and we are denoting by $|h(x)|$ the operator norm of the
matrix $h(x)$ and where
$|h'|=\displaystyle\sum_{|\alpha|=1}|\partial^{\alpha}h(x)|$,
$|h''|=\displaystyle\sum_{|\alpha|=2}|\partial^{\alpha}h(x)|$ and
$|h'''|=\displaystyle\sum_{|\alpha|=3}|\partial^{\alpha}h(x)|$.

Note also that these assumptions imply
\begin{equation}\label{detder}
\|\langle x\rangle^{1+\sigma}\grad (\sqrt{\det(h(x))})\|_{L^\infty}+\|\langle x\rangle^{1+\sigma}\grad (\sqrt{\det(h(x))}h^{jk}(x))\|_{L^\infty}\leq C_\grad
\end{equation}
 for some constant $C_\grad$ (the constant $C_\grad$ might be explicitly written in terms of  $C_h$, $\nu$ and $N$, but here we prefer to introduce another constant to keep notations lighter). 
 These are often referred to as \emph{long range} perturbations of the euclidean metrics.
\medskip

We are now ready to state our first result
\begin{theorem}\label{locsmoothteo}
Let $u$ be a solution to \eqref{diraceqcurv} with initial condition $u_0$, with $g$ satisfying \eqref{struct}, and assume that $h$ satisfies  \eqref{eq:assaabove}, \eqref{assasy} and \eqref{assdera} with the constants involved small enough. Then
for $ \eta_1,\eta_2 > 0$, there exists $C_{\eta_1,\eta_2} > 0$ independent from $u$ such that 
\begin{equation}\label{locsmoothest}
\| \langle x\rangle^{-3/2-\eta_1}u\|_{L^2_{\phi}L^2_x}+\| \langle x\rangle^{-1/2-\eta_2}\grad u\|_{L^2_{\phi}L^2_x}\leq C_{\eta_1,\eta_2}\|Hu_0\|_{L^2(\mathcal{M}_h)}^2.
\end{equation}
\end{theorem}
The spaces $L^2_\phi$, $L^2(\mathcal{M}_h)$ and $L^2(\mathcal{M}_g)$, which will be needed in the statement of Theorem \ref{locsmoothteo2}, are defined in a completely standard way (see forthcoming \eqref{Mf}, \eqref{Mg} and \eqref{Lphi}.
%

\begin{remark}Let us make a few comments about the functional framework which are valid for the two theorems. First, we work in geodesically complete geometries such that we have no issues with border terms when computing the virial identity that we use in the proof. Then, by $H^s$ norms, we mean norms that depend both on the spinorial aspect of the Dirac equation and the geometry. Namely, the $H^s$ norm of $u$ is defined as
$$
\|u\|_{H^s}^2 := \sum_{|\alpha|\leq s} \|\prod_{j=1}^3 D_j^{\alpha_j}u\|_{L^2(M_h)}^2
$$
where $(D_j)_{1\leq j\leq 3}$ are the covariant derivatives for Dirac bispinors (we define them later when introducing the Dirac equation), $\alpha = (\alpha_j)_{1\leq j\leq 3} \in \N^3$, $|\alpha| = \sum_{j} \alpha_j$ and $L^2(M_h)$ is defined thanks to the infinitesimal volume described by the metrics $h$.

In our settings, $\|H^2 \cdot\|_{L^2}$ is equivalent to the $H^2$ norm such that there is enough propagation of regularity of the Dirac equation to close the computation of the viral identity. 

Finally, let us remark that only in the setting of Theorem \ref{locsmoothteo}, the $H^s$ norm of $u$ is equivalent to the Sobolev norm in $\R^3$, at least up to $s=3$, thanks to the assumption of smallness and regularity of the perturbation.
\end{remark}

\begin{remark}
Our asymptotically flatness assumptions listed above are fairly standard in this setting (compare e.g. with \cite{cacdanluc}). The main example we have in mind is given by the choice $h_{jk}=(1+\varepsilon \langle x\rangle^{-\sigma})\delta_{jk}$ with $\sigma \in(0,1)$ and for some $\varepsilon$ sufficiently small: this matrix satisfies indeed all the assumptions of this subsection. As a further particular case we can think $h_{jk}$ to be a small and regular enough compactly supported perturbation of the flat metric. 
\end{remark}

%


\begin{remark} In fact, we can prove under assumptions of Theorem \ref{locsmoothteo} a slightly stronger version of estimate \eqref{locsmoothest}, namely the following
\begin{equation}\label{locsmoothcamp}
\|u\|_{XL^2_{\phi}}^2+\|\grad u\|_{YL^2_{\phi}}^2\leq C_{\nu,N,\sigma}\|Hu_0\|_{L^2(\mathcal{M}_h)}^2
\end{equation}
where the Campanato-type norms $X$ and $Y$ are defined at the end of this section, by the equations \eqref{defnormX}, \eqref{defnormY}. These spaces represent indeed somehow the natural setting when dealing with the multiplier method (see e.g. \cite{boussdanfan, cacdanluc}); nevertheless, we prefer to state our Theorem in this form for the sake of symmetry with the next result. We stress anyway that estimate \eqref{locsmoothcamp}, which is the one that we will prove, implies \eqref{locsmoothest}.
\end{remark}


\begin{remark}
As done in \cite{cacdanluc} for the Helmholtz equation, our proof allows us, after carefully following all the constants produced by the various estimates, to provide some explicit sufficient conditions that guarantee local smoothing estimate \eqref{locsmoothest}: we indeed quantify the closeness to a flat metric which we require by giving out explicit inequalities that the constants in Assumptions \eqref{assasy}-\eqref{assdera} must satisfy to get the result. These conditions are given in forthcoming subsection \eqref{conclusion}, by requiring the positivity of the constants $M_1$ and $M_2$ which reflects in smallness requirements on the constants $C_I$ and $C_h$ in \eqref{assasy}-\eqref{assdera}.
This fact, as mentioned, is strictly connected to the geometrical assumption of \emph{non-trapping} on the metric $g_{jk}$; therefore our strategy of proof gives, in a way, some explicit sufficient conditions that guarantee the metric $g$ to be non-trapping.
\end{remark}


\begin{remark}
In Minkowski space-time, the influence of a magnetic potential in equation \eqref{eq1} is reflected in the change of the covariant derivative, that is the substitution
$$
\grad \rightarrow \grad_A:=\grad-iA
$$
where $$A=A(x)=(A^1(x),A^2(x),A^3(x)): \R^3\rightarrow \R^3$$ is the magnetic potential. 
This phenomenon can be generalized to equation \eqref{diraceqcurv}: the presence of a magnetic potential has indeed essentially the effect of changing the covariant derivative $D_\mu$. In particular, by squaring the magnetic Dirac equation on a space with a metric $g_{jk}$ with the structure \eqref{struct} one obtains the following Klein-Gordon type equation
$$
-(\phi \partial_t)^2 u  + \tilde{\lap}_h u - \frac14 \mathcal R_h u -m^2 u -\frac12 F_{jk}[\gamma_i,\gamma_k]u= 0.
$$
where $\tilde{\lap}_h$ is the magnetic Laplace-Beltrami operator and $F_{jk}=\partial_j A_k-\partial_k A_j$ is the electromagnetic field tensor. The strategy of the present paper allows to deal with this more general situation: anyway, we prefer not to include magnetic potentials in order to keep our presentation more readable. We refer the interested reader to \cite{cacdanluc}, in which the electromagnetic Helmholtz equation is discussed with the same techniques as here. 
\end{remark}

\begin{remark}
The problem of proving Strichartz estimates for solutions to equation \eqref{diraceqcurv} seems significantly more difficult: variable coefficients perturbations indeed prevent the direct use of the standard Duhamel formula to handle the additional terms (see e.g. \cite{boussdanfan}) and requires a completely different approach involving phase space analysis and parametrices construction. We stress the fact, anyway, that proving Morawetz-type estimates (or local energy decay in the case of the wave equation) still represents a crucial step in this more involved setting, as they indeed provide a convenient space to place the errors of the parametrix. The interested reader should see \cite{metctat} and references therein. We also mention the fact that one could mimic the argument presented in \cite{burq}, where it is proved that global in time Strichartz estimates for solutions to the wave equation on a non flat background, and also outside of a compact obstacle, are implied by a suitable local smoothing 
estimate, provided the metric is assumed to be flat outside some ball and the solutions to \eqref{wavevar} which are compactly supported in space are known to satisfy local in time Strichartz estimates. This strategy seems to apply to our case, at least in order to obtain homogenous estimates, meaning that it could be adapted to equation \eqref{wavevar}, that presents an additional zero order term, and thus to \eqref{diraceqcurv}. This would give, at least, a conditional result. We intend to deal with all these problems in forthcoming papers.
\end{remark}

Next, we consider the specific case of the so called \emph{warped products}, that is metrics of the form \eqref{struct} with the additional structure
\begin{equation}\label{defWP}
h_{11} = 1, h_{1i} = h_{i1} =0 \textrm{ if } i\neq 1, h_{ij} = d(x^1) \kappa_{ij}(x^2,x^3)
\end{equation}
where $\kappa$ is a $2\times 2$ metric. We denote the scalar curvature of $\kappa$ by $\mathcal R_\kappa = \mathcal R_\kappa(x^2,x^3)$.

We will use the more comfortable (and intuitive) notation $r  = x^1$. In the case of the flat metric of $\R^3$, $d(r) = r^2$ and $\kappa$ is the metric of the sphere $S^2$. In all that follows, we assume that $\kappa$ is smooth enough ($\mathcal C^2)$ and that since $h$ should be positive, that $\kappa$ is a positive matrix.

We prove the following theorem. 

\begin{theorem}\label{locsmoothteo2}
Let $u$ be a solution to \eqref{diraceqcurv} with initial condition $u_0$, with $g$ satisfying \eqref{struct} and $h$ as in \eqref{defWP}. Then the following results hold.
\begin{itemize}
\item (Hyperbolic-type metrics). Take $d(r)=e^{r/2}$ in \eqref{defWP} and assume that for all $(x^2,x^3)$
$$\mathcal R_\kappa(x^2,x^3) > 0,\qquad m^2> \frac3{32}.$$
 Let $ \eta_1,\eta_2 > 0$. There exists $C_{\eta_1,\eta_2} > 0$  such that for all $u$ solution of the linear Dirac equation, we have 
\begin{equation}\label{locsmoothhyp}
\|e^{-r/4}\an{r}^{-(1+\eta_1)} u\|_{L^2(\mathcal{M}_g)}^2 +   \|e^{-r/4}\an{r}^{-(1/2+\eta_2)} \grad_h u\|_{L^2(\mathcal{M}_g)}^2 \leq C_{\eta_1,\eta_2}\|Hu_0\|_{L^2(\mathcal{M}_h)}.
\end{equation}

\item (Flat-type metrics).
Take $d(r)=r^2$ in \eqref{defWP} and assume that for all $(x^2,x^3)$, 
$$\mathcal R_\kappa \geq 2,\qquad m>0.$$
 Let $ \eta_1,\eta_2 > 0$. There exists $C_{\eta_1,\eta_2} > 0$  such that for all $u$ solution of the linear Dirac equation, we have 
\begin{equation}\label{locsmoothflat}
\|\an{r}^{-(3/2+\eta_1)} u\|_{L^2(\mathcal{M}_g)}^2 +   \|\an{r}^{-(1/2+\eta_2)} \grad_h u\|_{L^2(\mathcal{M}_g)}^2 \leq C_{\eta_1,\eta_2} \|Hu_0\|_{L^2(\mathcal{M}_h)}.
\end{equation}
\item (Sub-flat type metrics) Take $d(r)=r^n$ in \eqref{defWP} with $n\in ]2-\sqrt 2,4/3]$. There exists $C_n> 0$ such that if for all $(x^2,x^3)$, $\mathcal R_\kappa \geq C_n$, then for all $ \eta_1,\eta_2 > 0$, there exists $C_{\eta_1,\eta_2,n} > 0$  such that for all $u$ solution of the linear Dirac equation, we have 
\begin{equation}\label{locsmoothint}
\|\an{r}^{-(3/2+\eta_1)} u\|_{L^2(\mathcal{M}_g)}^2 +   \|\an{r}^{-(1/2+\eta_2)} \grad_h u\|_{L^2(\mathcal{M}_g)}^2 \leq C_{\eta_1,\eta_2,n} \|Hu_0\|_{L^2(\mathcal{M}_h)}.
\end{equation}

\end{itemize}
\end{theorem}

\begin{remark} The hypothesis on the mass comes from the fact that the infimum of the curvature $\inf\mathcal R_h$ is not positive (it is negative in the hyperbolic case, and $0$ in the other ones). This issue arises when one estimates the $ H^1$ norm of $u$ with the energy, the mass is then used to compensate this negative curvature as it is chosen such that $\inf \mathcal R_h + 4m > 0$. Note that proving Hardy's inequality for any $\kappa$ would permit to have $m=0$ in the flat case. \end{remark}

\begin{remark} There are different conditions for the curvature of $\kappa$. One reason is specified in the previous remark : we need the curvature not to be infinitely small, such that the energy controls the $H^1$ norm. This explains $\mathcal R_\kappa \geq 0$ in the subflat case and $\mathcal R_\kappa \geq 2$ in the flat case. What is more, we use the curvature term to compensate losses due to the bi-Laplace-Beltrami term in the virial identity in the hyperbolic and subflat cases. This is why we have more constraining hypothesis.\end{remark}


%

\begin{remark} By mixing the techniques used to prove Theorem \ref{locsmoothteo} and Theorem \ref{locsmoothteo2}, one should expect to be able to prove local smoothing for the Dirac equation on metrics that are asymptotically like the warped products we presented. It is also reasonable to expect, simply by mimicking techniques, the same results to hold for a more generic warped product. In the cases, we present, we always take a multiplier of the same form. The computations for a generic $d$ are somewhat tedious, but if one wishes to repeat the argument for another $d$, one natural assumption is $\frac{d(r)}{d'(r)}\lesssim \an r$. The rest would be finding sufficient hypothesis on $\lap_h^2 r$ and $\lap_h^2 r^2$.\end{remark}

The strategy for proving these results relies on the multiplier method: using some standard integration by parts techniques we will be able to build a proper virial identity for equation \eqref{diraceqcurv} (see Proposition \ref{prop-exptheta}) which, by choosing suitable multiplier functions, will allow us to prove local smoothing estimates. Notice that the indefinite sign of the Dirac operator provides a major obstacle in the application of the multiplier method directly to equation \eqref{diraceqcurv}; this is the reason why one resorts on the squared equation \eqref{wavevar}, as done in the magnetic case in \cite{boussdanfan}. We also stress the fact that such a method does not seem to apply in lower dimensions due to the difficulty of finding proper multipliers, and to the best of our knowledge we are not aware of similar results in dimensions $1$ or $2$. On the other hand, the method would be well adapted to deal with high dimensional frameworks: anyway, the extension of the Dirac equation to high dimensions  is not quite as straightforward as, e.g, the Schr\"odinger or wave ones, and it would require some additional work and a fair amount of technicalities that we prefer not to deal with here. Also, the high dimensional cases seems to present a relatively scarce relevance 
in the applications. This is the main reason behind our restriction to dimension $n=3$. 

The plan of the paper is the following: in section \ref{curv} we review the theory of Dirac operators on curved spaces, showing how to properly build a dynamical equation, in section \ref{virial} we prove the virial identity that is the crucial stepping stone for local smoothing with the use of the multiplier method, while sections \ref{locsmot1} and \ref{locsmot2} are devoted, respectively, to the proofs of Theorems \ref{locsmoothteo} and \ref{locsmoothteo2}.

\vskip0.5cm

{\bf Notations}

We conclude this section by fixing some notations that will be adopted throughout the paper together with some elementary properties. Some of the definitions will be anyway recalled when needed to help the reader's reading.
Let $h=h(x)$ be a $3\times3$, positive definite, real matrix that defines, in a standard way, a metric tensor. We recall that the \emph{scalar curvature} can be written as
\begin{equation}\label{scalcurv}
\mathcal R_h=h^{jk}\left(\frac{\partial}{\partial x_i}\Lambda^i_{jk}-\frac{\partial}{\partial x_k}\Lambda_{ji}^i+\Lambda_{jk}^\ell\Lambda_{i\ell}^i+\Lambda_{ji}^\ell\Lambda^i_{k\ell}\right),
\end{equation}
where $\Lambda^i_{jk}$ denote the standard Christoffel symbols (we use will $\Gamma$ for the ones associated to $g$).

 In what follows we will use the compact notation for the matrices
$$
h=h(x)=[h_{jk}(x)]_{j,k=1}^3\quad
h_{inv}=h_{inv}(x)=[h^{jk}(x)]_{j,k=1}^3.
$$
We will need the quantities
$$
\hat{h}(x)=h^{jk}\hat{x}_j\hat{x}_k,\qquad \overline{h}(x)={\rm Tr}(h_{inv}(x))=h^{kk}(x)
$$
where we are using the standard conventions for implicit summation and $\hat{x}=x/|x|$. Notice that, as $h(x)$ is assumed to be positive definite, 
$$0\leq\hat{h}(x)\leq \overline{h}(x)$$ for every $x$. Also, we will use the compact notation
$$
\tilde{h}^{jk}=\sqrt{\det(h)}h^{jk}.
$$
Straightforward computations show that, for every sufficiently regular radial function $\psi$,
\begin{eqnarray}\label{lappsi}
\Delta_h\psi(x)&=&
\hat{h}\psi''+\frac{\overline{h}-\hat{h}}{|x|}\psi'+\frac{1}{\sqrt{\det(h)}}\partial_j (\tilde{h}^{jk})\hat{x}_k\psi'
\end{eqnarray}
where $'$ denotes the radial derivative and we are slightly abusing notations by identifying the functions $\psi(x)$ and $\psi(|x|)$.

We now introduce the notation for the functional spaces we are using. The norms with respect to time are  given by
\begin{equation}\label{Lphi}
\|u\|_{L^2_{\phi,T}}^2=\int_0^T\frac{|u(t)|^2}{\phi(t)}dt,\qquad 
\|u\|_{L^2_{\phi}}^2=\int_0^{+\infty}\frac{|u(t)|^2}{\phi(t)}dt
\end{equation}
where $T>0$ and $\phi$ is the positive function given in the definition of $g$ \eqref{struct}.
In particular, when $\phi=1$ these norms recover the standard $L^2_T$ (resp. $L^2$) ones, and we will simply denote with $L^2_T=L^2_{1,T}$. We shall use freely either $\partial_0$ or $\partial_t$ for the time derivative.

The norms $\|\cdot\|_{L^2(\mathcal{M}_g)}$ and $\|\cdot\|_{L^2(\mathcal{M}_h)}$ are respectively the $L^2$ norms on the manifold $\mathcal{M}_g$ and $\mathcal{M}_h$, that is
\begin{equation}\label{Mg}
\|f\|_{L^2(\mathcal{M}_g)}^2 =\int_{\mathcal{M}_g} |f|^2=  \int_{\R\times D(h)} |f(t,x)|^2 \sqrt{\det(g(t,x))} dxdt
\end{equation}
and \begin{equation}\label{Mf}
\|f\|_{L^2(\mathcal{M}_h)}^2 =\int_{\mathcal{M}_h} |f|^2 = \int_{D(h)} |f(x)|^2 \sqrt{\det(h(x))} dx
\end{equation}
where $D(h)$ is the set where $h$ is defined. 
Due to the structure of $g$ \eqref{struct}, we have $g(t,x) =- \phi^{-2}(t) h(x)$, which yields
$$
\|f\|_{L^2(\mathcal{M}_g)}^2 = \int_{\R} \|f(t,\cdot)\|_{L^2(\mathcal{M}_h)}^2 \phi^{-1}(t) dt.
$$
Concerning the scalar products, we will denote the scalar product induced by $h$ as
$$
\langle f, g \rangle_h = \int_{\mathcal{M}_h} fg = \int_{D(h)} \overline{f(x)}g(x) \sqrt{\det(h)} d^3 x,
$$
while with $\grad_h f \cdot \grad_h$ we mean the operator $h^{ij} \an{D_i f, D_j \cdot }_{\C^4}$. 

In the asymptotically flat case we will also make use of the so called \emph{Campanato-type norms} , which are defined as (note that $\langle R\rangle=\sqrt{1+R^2}$)
\begin{equation}\label{defnormX}
\|v\|_{X}^2:=\sup_{R>0}\frac1{\langle R\rangle^2}\int_{\mathcal{M}_h\cap S_R}|v|^2 dS
=\sup_{R>0}\frac1{\langle R\rangle^2}\int_{S_R}|v|^2 \sqrt{\det(h)}dS
\end{equation}
where $dS$ denotes the surface measure on the surface of the ball $\{|x|=R\}$, and
\begin{equation}\label{defnormY}
\|v\|_{Y}^2:=\sup_{R>0}\frac1{\langle R\rangle}\int_{\mathcal{M}_h\cap B_R}|v|^2 dx=
\sup_{R>0}\frac1{\langle R\rangle}\int_{B_R}|v|^2 \sqrt{\det(h)}dx
\end{equation}
where we are denoting by $S_R$ and $B_R$, respectively, the surface and the interior of the sphere of radius $R$ centred in the origin. Notice that $\|\cdot\|_Y$ is equivalent to the norm whose square is
\begin{equation}\label{1norm}
\sup_{R\geq1}\frac1R\int_{\mathcal{M}_h\cap B_R}|v|^2.
\end{equation}



\vskip0.7cm
{\bf Acknowledgments.} The authors are grateful to Paolo Antonini and Gianluca Panati for useful discussions on the topic. The first named author is supported by the FIRB 2012 "Dispersive dynamics, Fourier analysis and variational methods" funded by MIUR (Italy) and by National Science Foundation under Grant No. DMS-1440140 while the author was in residence at the Mathematical Sciences Research Institute in Berkeley, California, during the Fall 2015 semester. 




\section{From vierbein to dreibein}\label{curv}

Our aim in this subsection is first of all to give some (basic) motivations and backgrounds that lead to the study of the Dirac equation on a non-flat setting, and then to describe the Dirac equation on curved space-time in the case of a metrics which dissociates time and space. We prove indeed that it can be written as
$$
i\phi\partial_t  u - Hu = 0
$$
with $\phi$ is the function appearing in the metrics \eqref{struct} and $H$ an operator such that
$$
H^2 = - \lap_h +\frac14\mathcal R_h + m^2
$$
where $h$ is the space metrics, $\lap_h$ the Laplace-Beltrami operator associated to this metrics, $m \in \R^+$ is a parameter (a mass) and $\mathcal R_h$ the scalar curvature.

\subsection{Motivation and first construction: vierbein}

The study of the Dirac equation in curved space-time is part of the more general subject of Quantum Field Theory (later referred to as QFT) in curved space-time. One of the successes of this theory is the description of entropic black holes. Its field of investigation is the description of elementary particles at energy below the Planck constant. The Dirac equation is part of it as it models the dynamics of relativistic electrons. What is underlying this approximation or reconciliation between QFT and general relativity is that the effects of the dynamics of particles should either be included in the model for the metrics via appropriate couplings or simply neglected in which case the metrics is described outside the studied system of particles.

The idea of QFT in curved space-time is the following. As in classical QFT (by which we mean QFT in a Minkowski space-time), the laws of physics should be independent from the choice of coordinates or referential (here one may understand it as the choice of a frame for $T_{x_0}M$, that is the equations derived from this theory should be covariant. To derive appropriate models for elementary particles, Cartan's formalism is used. In other words, one introduces a $n$-bein, or in our $1+3$ dimensional case a vierbein. A vierbein may be seen as matrices $e^a_\mu(x)$ depending on a point $x$ of the manifold $\mathcal{M} = (\R^{1+3}, g)$. They satisfy
$$
e_a^\mu(x)g_{\mu\nu}(x)e_b^\nu(x) = \eta_{ab} \Leftrightarrow e^\mu_a (x) \eta^{ab} e^{\nu}_b(x) = g^{\mu \nu}(x)
$$
where $\eta$ is the Minkowski metrics. i.e.
$$
\eta^{ab} = \left \lbrace{\begin{array}{ll}
1 & \textrm{ if } a=b=0 \\
-1 & \textrm{ if } a=b\neq 0 \\
0 & \textrm{ otherwise. }\end{array}} \right. 
$$
 Note that $e^a_\mu$ is not uniquely defined as $(e')^a_\mu = L^a_b(x) e^b_\mu(x)$ where $L^a_b \in SO(1,3)$ also satisfies the same equations. The vierbein is what links $\mathcal{M}$ to the flat space-time or more precisely their tangeant space. Indeed, by writing $y^a= e^a_\mu(x_0) z^\mu$, one gets $y^a\eta_{ab}y^b = z^\mu g_{\mu\nu}(x_0)z^\nu$, hence $e_a^\mu$ sends the tangent space $T_{x_0}\mathcal{M}$ to a Minkowski space-time in a way that preserves the inner product. Note that the matrix $e(x)$ should be reversible. Changing $e$ into $e'$ induces a change of referential and variable that satisfies $ e_a^\mu(x_0) \partial_\mu = (e')_a^\mu(x_0) \partial'_\mu $ and hence preserves the inner product. Conversely, a change of variables or referential that preserves the inner product may be seen as a change of vierbein. Note that this change depends on the point $x_0$ of the manifold, as the matrix $L^a_b$ depends on the point of the manifold. Hence, a change of variable or of referential that preserves the inner product of the tangent 
spaces is described by a matrix field $L^a_b\in SO(1,3)$ and eventually a translation, making it a Lorentz group.

Now that we have introduced this point of view for changes of variable/ refrential, or more precisely covariance, we can introduce the notion of spinors. Let $u$ be a vector field over $\mathcal{M}$. Assume that we have chosen some coordinates such that $u = u(x)$ and a vierbein $e$. We want to pass from $e$ to $e'$, where the change of vierbein is induced by $L$ (and perhaps a translation). For vectors, we get $u'(x') = u(x)$. But in the case of spinors $\psi$, this transform works such that the equations satisfied by $\psi$ are independent from the choice of coordinates and of vierbein. In general, this transform writes $\psi'(x') = S(L(x))\psi(x)$ where $S$ is a group representation of $SO(1,3)$. If $S$ is trivial we retrieve the change of variable for vectors. This is analogous to what we have in the case of classical QFT. 

We define now the notion of covariant derivative for spinors. Let us recall that the covariant derivative of a vector field $u$ is given by
$$
D_\mu u_\nu = \partial_\mu u_\nu + \Gamma_{\mu\nu}^\sigma u_\sigma
$$
and is by definition independent from the choice of coordinates, that is
$$
D_\mu' u'_\nu = \frac{\partial (x')^\rho}{\partial x^\mu}D_\rho u_\nu.
$$
In the case of a spinor $\psi$, we write 
\begin{equation}\label{defcov}
D_\mu \psi = \partial_\mu \psi + B_\mu (x) \psi
\end{equation}
where $B_\mu$ is a field over $\mathcal{M}$ which lies in the Lie algebra generated by $S$ and which is to be determined. We require that $D_\mu$ is independent from the choice of variable, that is
$$
D_\mu'\psi' = \frac{\partial (x')^\rho}{\partial x^\mu} S(L) D_\rho \Psi.
$$
Write 
$$
F(L) = \psi \mapsto D_\mu' \psi' - \frac{\partial (x')^\rho}{\partial x^\mu} S(L) D_\rho \Psi.
$$
Note that $F(L)$ belongs to the Lie algebra induced by the representation $S$. Because $SO(1,3)$ is connected, it is sufficient to have that $B_\mu$ satisfies $F(Id) = 0$ and the differential of $F$ at the identity is $0$, to have that $F$ is identically $0$. In other words, write $L^a_b = \delta^a_b + \varepsilon^a_b(x)$ where $\delta$ is the Kronecker symbol, write $F(L) $ at first order in $\varepsilon$ and chose $B$ such that this first order is null for all suitable variations $\varepsilon$. We do not wish to repeat the computations of \cite{parktoms} but one may find them at pages 221-229. These computations yield
$$
B'_\mu(x) = B_\mu(x) + i\varepsilon^{ab}(x)[\Sigma_{ab},B_\mu(x)] -i\partial_\mu \varepsilon^{ab} \Sigma_{ab}
$$
where $\Sigma_{ab}$ are generators of the Lie algebra induced by $S$ and $[\cdot, \cdot]$ is the commutator. At first order $S(1+ \varepsilon^a_b)  = 1+i\varepsilon^{ab}\Sigma_{ab}$. Because $B_\mu$ belongs to this Lie algebra, we have
$$
B_\mu(x) = B_\mu^{ab}(x)\Sigma_{ab}.
$$
The previous equation on $B_\mu$ yields a system of equation on $B_\mu^{ab}$ which is solved by $B_\mu^{ab} = i\omega_\mu^{ab}$ where $\omega$ is the spin connection :
$$
\omega_\mu^{ab} = e^a_\nu \partial_\mu e^{\nu b} + e^a_\nu \Gamma_{\mu \sigma}^\nu e^\sigma_b
$$
where $\Gamma$ is the affine connection given by
$$
\Gamma_{\mu \sigma}^\nu = \frac12 g^{\nu \lambda}(\partial_\mu g_{\lambda \sigma}+ \partial_\sigma g_{\mu \lambda} - \partial_\lambda g_{\mu \sigma}).
$$

Let us focus on the Dirac equation. It is written in analogy with the Dirac equation in the Minkowski space-time
$$
(\mathcal{D}\gamma^a e_a^\mu D_\mu - m)\psi = 0
$$
where $\gamma^a$ are the usual Dirac matrices, i.e. $\gamma^0 = \beta$ and $\gamma^i = \gamma^0 \alpha_i$ for $i\in \{1,2,3\}$. In particular $[\gamma^a,\gamma^b] = \eta^{ab}$. Writing $\underline{\gamma^\mu} = e^{\mu}_a \gamma^a$, the Dirac equation writes
$$
(\underline{\gamma^\mu} D_\mu - m)\psi = 0
$$
with $[\underline{\gamma^\mu},\underline{\gamma^\nu}] = g^{\mu\nu}$. It remains to specify $S$ or $\Sigma_{ab}$ such that this equation is independent from the choice of coordinates and vierbein. We still work at first order and we get that for the independence to be satisfied, we require
$$
S(L)\gamma^a S^{-1}(L) L_a^b  = \gamma^b.
$$
By replacing $L_a^b$ by $1+ \varepsilon_a^b$ and $S(L)$ by $1+ i\varepsilon^{ab}\Sigma_{ab}$ we get a system of equations on $\Sigma$ which is solved by $ \Sigma_{ab} = -\frac{i}{8}[\gamma_a,\gamma_b]$.

In the end, we get that the covariant derivative for a Dirac spinor is given by
\begin{equation}\label{covvier}
D_\mu = \partial_\mu + \frac18 \omega^{ab}_\mu [\gamma_a,\gamma_b]
\end{equation}
and that the Dirac equation built in such a way is independent from the choice of variables and referential, or in other words, relativistically invariant. One important fact to be noticed is that, by construction, we have
$$
\mathcal D^2 = - \square_g-\frac14 \mathcal R_g
$$
where $ \mathcal R_g$ is the scalar curvature associated to the metric $g$ and $\square_g$ is the d'Alembertian associated to the metric $g$, that is, $\square_g = D^jD_j $.

\subsection{Dreibein}

In this subsection, we prove that if time and space are decorrelated in the metrics, then they also are in the Dirac equation. The idea is then the following, having a metrics $g$ of the form
$$
g = \begin{pmatrix} 1 & 0 \\ 0 & -h\end{pmatrix}
$$
where $h$ is positive, we write the dreibein, that is the connection between $h$ and a Euclidean space of dimension $3$, the affine connection, the spin connection and the covariant derivatives relative to $h$, and explain how they relate to the vierbein, affine connection, the spin connection and the covariant derivative relative to $g$. Then, we write the Dirac equation with the help of the information on $h$, which helps us disconnect time and space as in 
$$
i\phi \partial_t u = Hu
$$
with $H = -\gamma^0 (if^\mu_a \gamma^a D_\mu+m)$ where $f^\mu_a$ is a dreibein for $h$, $\gamma^a$ are the standard Dirac matrices, and $D_\mu$ is the covariant derivative for spinors in $\R^3,h$.  We must say that the result is the natural one, and that this subsection is preeminently technical.

We consider a metric $g$ of the following form
$$
g_{\mu\nu} = \left \lbrace{\begin{array}{ll}
\phi^{-2} (x^0) & \textrm{ if } \mu= \nu = 0\\
0 & \textrm{ if } \mu\nu = 0 \textrm{ and } \mu\neq \nu\\
-h_{\mu\nu}(\overrightarrow x) & \textrm{ otherwise. }\end{array}} \right.
$$
where $\overrightarrow x = (x^1,x^2,x^3)$.

Note that in the sequel we will use the latin letters a,b, etc... for the Minkowski space $\R^{1+3},\eta$ or for the Euclidean space $\R^3$, the latin letters i,j, etc... for the space $D(h),h$ (where $D(h)$ is the space where $h$ is defined) and the greek letters $\mu,\nu$, etc... for the space, $\mathcal M, g$. 

Let $f_a^i$ be a dreibein hence satisfying
$$
h^{ij} = f_a^i\delta^{ab}f_b^j
$$
where $\delta$ here denotes the Kronecker symbol. In this sum, $a$ and $b$ are taken only between $1$ and $3$. Note that we can and do choose $f$ independent from $x^0$.

In the sequel, we write $e^a_\mu$ a vierbein for $g$, $\Gamma_{\mu\nu}^\sigma$ the affine connection for $g$, while $\Lambda_{ij}^k$ is the affine connection for $h$, $\omega_\mu^{ab}$ is the spin connection for $g$, and $\alpha_i^{ab}$ is the one for $h$. 

\begin{proposition}\label{prop-bein} Write 
$$
e_a^\mu = \left \lbrace{\begin{array}{ll}
\phi (x^0) & \textrm{ if } \mu= a = 0\\
0 & \textrm{ if } \mu a = 0 \textrm{ and } \mu\neq a\\
f_a^\mu & \textrm{ otherwise. }\end{array}} \right.
$$
The matrix $e_a^\mu$ is a vierbein for $g$.
\end{proposition}

\begin{proof} The issue is to prove that $e_a^\mu \eta_{ab} e^b_\nu = g_{\mu \nu}$. We start with $\mu = \nu = 0$.
We have
$$
e_a^0 \eta^{ab}e_b^0 = \phi^2 \delta_a^0 \eta^{ab} \delta_b^0 = \phi^2 \eta^{00} = g^{00}.
$$

What is more, with $i\neq 0$ ($\mu = 0$, $\nu  = i \neq 0$),
$$
e_a^0 \eta^{ab}e_b^i = \phi e_0^i = 0= g^{0i}
$$
and for the same reason $e_a^i \eta^{ab}e_b^0 = g^{i0}$.

And finally, with $ij\neq 0$,
$$
e_a^i \eta^{ab}e_b^j = f_a^i \eta^{ab} f_b^j = f_a^i (-\delta^{ab}) f_b^j = -h^{ij} = g^{ij}.
$$

This makes $e_a^\mu$ a suitable vierbein for $g$. \end{proof}

Let us see how the Christoffel symbol is changed.

\begin{proposition}\label{prop-Christoffel}
Let 
$$
\Lambda_{ij}^k = \frac12 h^{kl} ( \partial_i h_{lj} + \partial_j h_{il} - \partial_l h_{ij}).
$$
We have 
$$
\Gamma_{\mu \nu}^\sigma = \left \lbrace{\begin{array}{ll}
-\phi^{-1} \phi' & \textrm{ if } \mu = \nu = \sigma = 0 \\
\Lambda_{\mu \nu}^\sigma & \textrm{ if } \mu\nu\sigma \neq 0\\
0 & \textrm{ otherwise }.
\end{array}} \right.
$$
\end{proposition}

\begin{proof}
We have 
$$
\Gamma_{\mu \nu}^0 = \frac12 g^{0\lambda} (\partial_\mu g_{\lambda \nu} + \partial_\nu g_{\mu \lambda} - \partial_\lambda g_{\mu \nu}).
$$
Since $g^{0\lambda} = 0$ if $\lambda \neq 0$, we get
$$
\Gamma_{\mu \nu}^0 = \frac12 \phi^{2} (\partial_\mu g_{0 \nu} + \partial_\nu g_{\mu 0} - \partial_0 g_{\mu \nu}).
$$
Assume $\nu \neq 0$. We have that $g_{0\nu} = g_{\nu 0} = 0$. Since $g_{00}$ depends only on $x^0$, and $g_{\mu 0} = 0$ if $\mu \neq 0$, we have $\partial_\nu g_{\mu 0} = 0$. Since $h$ does not depend on $x^0$, we have $\partial_0 g_{\mu\nu}=0$.  This yields
$$
\Gamma_{\mu \nu}^0 = 0
$$
if $\nu\neq 0$ and by symmetry, $\Gamma^0_{\mu \nu} = 0$ if $\mu \neq 0$. 

Besides, 
$$
\Gamma_{00}^0 = \frac12 \phi^2 \partial_0 g_{00} = -\phi^{-1} \phi'.
$$

We have considered all the cases when $\sigma =0$. Now we assume $\sigma \neq 0$, and we consider all the cases when $\mu = 0$. We have
$$
\Gamma_{0\nu}^\sigma = \frac12 g^{\sigma\lambda} (\partial_0 g_{\lambda \nu} + \partial_\nu g_{0 \lambda} - \partial_\lambda g_{0 \nu}).
$$
Since $\sigma \neq 0$, the sum over $\lambda$ is only for $\lambda \in \{1,2,3\}$. Hence, replacing $\lambda $ by $i$
$$
\Gamma_{0\nu}^\sigma = -\frac12 h^{\sigma i} (\partial_0 g_{i \nu} + \partial_\nu g_{0 i} - \partial_i g_{0 \nu}).
$$
Now it appears that $\partial_\nu g_{0 i} = \partial_i g_{0 \nu} = 0$. Therefore,
$$
\Gamma_{0\nu}^\sigma = -\frac12 h^{\sigma i} \partial_0 g_{i \nu}.
$$
If $\nu = 0$, $g_{i\nu} = 0$ and otherwise it does not depend on $x^0$, hence
$$
\Gamma_{0\nu}^\sigma = 0
$$
and by symmetry
$$
\Gamma_{\mu 0}^\sigma = 0.
$$

Finally, if $\mu\nu \sigma \neq 0$, we may write $\mu = i$, $\nu = j$, $\sigma = k$. We have
$$
\Gamma_{ij}^k = \frac12 g^{k\lambda}(\partial_i g_{j\lambda} + \partial_j g_{i\lambda} - \partial_\lambda g_{ij}).
$$
Because of the form of $g$, the sum over $\lambda$ is taken only for $\lambda \in \{1,2,3\}$, we replace $\lambda$ by $l$, and get
$$
\Gamma_{ij}^k = \frac12 h^{kl}(\partial_i h_{jl} + \partial_j h_{il} - \partial_l h_{ij})= \Lambda_{ij}^k.
$$

Therefore, we retrieve the result :
$$
\Gamma_{\mu \nu}^\sigma = \left \lbrace{\begin{array}{ll}
-\phi^{-1} \phi' & \textrm{ if } \mu = \nu = \sigma = 0 \\
\Lambda_{\mu \nu}^\sigma & \textrm{ if } \mu\nu\sigma \neq 0\\
0 & \textrm{ otherwise }.
\end{array}} \right.
$$
\end{proof}

Let us see how the spin connection is changed.

\begin{proposition}\label{prop-spinconnection}
Let 
$$
\alpha_i^{ab} = f_j^a \partial_i f^{jb} + f_j^a \Lambda_{ik}^j f^{kb}.
$$
We have
$$
\omega_{\mu }^{ab} = \left \lbrace{\begin{array}{ll}
\alpha_\mu^{ab} & \textrm{ if } \mu ab \neq 0\\
0 & \textrm{ otherwise }.
\end{array}} \right.
$$
\end{proposition}

\begin{proof}
Indeed, we have 
$$
\omega_0^{ab} = e_\nu^a \partial_0 e^{\nu b}+ e_\nu^a \Gamma_{0\sigma}^\nu e^{\sigma b}.
$$
Because $\Gamma_{0\sigma}^\nu = 0$ if $\sigma + \nu \neq 0$ we have $e_\nu^a \Gamma_{0\sigma}^\nu e^{\sigma b} = e_0^a \Gamma_{00}^0 e^{0b}$ which, since $e^{00} = \eta^{0a}e_{a}^0= \phi$, is equal to $-\phi \phi'$ if $a=b=0$ and to $0$ otherwise.

If $\nu + b\neq 0$, then either $e^{\nu b} = e^{0 b} = 0$ since $b \neq 0$, or $e^{\nu b} = e^{\nu 0} = 0$ or $e^{\nu b} = f^{\nu b}$ and does not depend on $x^0$.

Hence, $\partial_0 e^{\nu b} = 0$ if $\nu+b\neq 0$ and we have $\partial_0 e^{00}=\partial_0 \phi =  \phi'$ we get $e_\nu^a \partial_0 e^{\nu b} = \phi \phi'$ if $a=b=0$ and to $0$ otherwise. Therefore,
$$
\omega_0^{ab} = 0.
$$

We have considered all the cases where $\mu = 0$. We assume $\mu \neq 0$. We deal with the case $a=0$.

We have
$$
\omega_\mu^{0b} = e_\nu^0 \partial_\mu e^{\nu b}+ e_\nu^0 \Gamma_{\mu\sigma}^\nu e^{\sigma b} = \phi \partial_\mu e^{0 b}+ \phi\Gamma_{\mu\sigma}^0 e^{\sigma b}.
$$
Since $\mu \neq 0$, $\Gamma_{\mu\sigma}^0 = 0$. Since $e^{0b}$ depends only on $x^0$ and $\mu \neq 0$, $\partial_\mu e^{0 b}=0$.

Since $\omega_\mu^{ab} = -\omega_\mu^{ba}$, we have that $\omega_\mu^{a0} = 0$.

We have dealt with all the cases where either $a$, $b$ or $\mu$ is equal to $0$. We now treat the case $\mu ab \neq 0$.

We can replace the sums on the greek letters by sums on latin letters, this yields
$$
\omega_i^{ab} = f^a_j \partial_i f^{jb} + f^a_j \Gamma^j_{ik} f^{kb} =  f^a_j \partial_i f^{jb} + f^a_j \Lambda^j_{ik} f^{kb} = \alpha_i^{ab}.
$$
This gives the result.

\end{proof}

\subsection{Covariant derivative and Dirac operator}

The covariant derivative is given by
$$
D_0 = \partial_0 \, ,\, D_i = \partial_i + \frac18 \alpha_i^{ab} [\gamma_a,\gamma_b].
$$
Therefore, the Dirac operator can be written as
$$
\mathcal D = i\gamma^0 \phi \partial_0 + i \gamma^a f_a^j D_j.
$$
Let 
$$
\mathcal H = i \gamma^a f_a^j D_j \textrm{ and } H = -\gamma^0 (\mathcal H + m).
$$

\begin{proposition}\label{prop-sqH} With these notations, we have
\begin{equation}\label{sqH} 
H^2 = m^2-\lap_h + \frac14 \mathcal R_h.
\end{equation}
\end{proposition}

\begin{proof}
First, we prove that $\mathcal H^2 = \lap_h - \frac14 \mathcal R_h$.

We have 
$$
\mathcal D^2 = \mathcal H^2 + (i\gamma^0 \phi \partial_0)^2 + (i\gamma^0 \phi \partial_0 \mathcal H + \mathcal H i\gamma^0 \phi \partial_0).
$$
Since $\gamma^0$ commutes with $\phi$ and $\partial_0$ and $(\gamma^0)^2 = 1$, we have $(i\gamma^0 \phi \partial_0)^2 = - (\phi \partial_0)^2$.

Since $\phi$ and $\partial_0$ commute with $\mathcal H$ and $\gamma^0$, we have 
$$
(i\gamma^0 \phi \partial_0 \mathcal H + \mathcal H i\gamma^0 \phi \partial_0) = i\phi \partial_0 ( \gamma^0 \mathcal H  + \mathcal H \gamma^0).
$$
Given the Dirac matrices \eqref{diracmatrices} (recall, again that $\gamma^0 = \beta$ and $\gamma^i = \gamma^0 \alpha_i$) , we have for all $a>0$, $\gamma^0\gamma_a = -\gamma_a\gamma^0$. Hence, $\gamma^0$ commutes with $[\gamma_a,\gamma_b]$ and thus with $D_j$. Therefore, we obtain
$$
( \gamma^0 \mathcal H  + \mathcal H \gamma^0) = i(\gamma^0 \gamma_a + \gamma_a\gamma^0)f_a^j D_j = 0.
$$
We get
$$
\mathcal D^2 = \mathcal H^2 - (\phi \partial_0)^2.
$$
We recall that $\mathcal D^2 = -\square_g - \frac14 \mathcal R_g$ and given the metric $\square_g = (\phi \partial_0)^2 - \lap_h $ and $\mathcal R_g = \mathcal R_h$. Therefore,
$$
\mathcal H^2 = \lap_h - \frac14 \mathcal R_h .
$$

Finally, we have 
$$
H^2 = \gamma^0( \mathcal H +m)\gamma^0 (\mathcal H +m) 
$$
and since $m$ commutes with $\gamma^0$ and $(\gamma^0)^2 = Id$ we get
$$
H^2= (\gamma^0 \mathcal H \gamma^0 + m)(\mathcal H + m)
$$
and since $\mathcal H$ anti-commutes with $\gamma^0$ and commutes with $m$, we get
$$
H^2 = (-\mathcal H + m) (\mathcal H +m)= m^2- \mathcal H^2 = m^2-\lap_h + \frac14 \mathcal R_h.
$$
\end{proof}
Besides, notice that
$$
m+\mathcal D = \gamma^0 (i\phi \partial_0 - H)
$$
and 
$$
(i\phi \partial_0 + H)(i\phi \partial_0 - H) = - (\phi \partial_0)^2 - H^2 = - (\phi \partial_0)^2+ \lap_h - \frac14 \mathcal R_h+m^2.
$$
%

\begin{corollary}If $u$ solves the Dirac equation
\begin{equation}
i\phi \partial_t u - H u = 0
\end{equation}
then $u$ satisfies also
\begin{equation}
-(\phi \partial_t)^2 u  + \lap_h u - \frac14 \mathcal R_h u -m^2 u = 0.
\end{equation}
\end{corollary}

\begin{remark}\label{rem-phione}
After a change of variable, we can replace $\phi$ by $\phi = 1$, and we get the more simple expression,
$$
m+ \mathcal D = \gamma^0 (i \partial_t - H),
$$
and 
$$
(i \partial_t + H)(i \partial_t - H) = - \partial_t^2 + \lap_h - \frac14 \mathcal R_h - m^2.
$$
\end{remark}

\section{Virial Identity}\label{virial}

We consider the linear equation 
\begin{equation}\label{diraceq}
i\phi \partial_t u - H u = 0.
\end{equation}

We have seen that if $u$ solves \eqref{diraceq}, then $u$ also solves 
\begin{equation}\label{waveeq}
(\phi \partial_t)^2 u + Lu =0
\end{equation}
with $L= H^2 = \frac14 \mathcal R_h+m^2 - \lap_h$. 
Note that $L$ is self-adjoint for the inner product $\langle \cdot,\cdot\rangle_h$. Let us prove quickly that $L$ is indeed symmetric. Take $u,v$ test-functions (smooth with compact support avoiding possible coordinate singularities of $B_j$). The issue comes from $-\lap_h$. Let us prove formally that
$$
\an{u,\lap_h v}_h = -\an{\grad_h u,\grad_h v}_h.
$$
We have
$$
\lap_h v =D^jD_j v = \tilde\lap_h v + B^i \partial_i v + \tilde D^i B_i v + B^iB_i v
$$
where $\tilde  \lap_h$ is the Laplace-Beltrami operator for scalars, $\tilde D^i \Psi_k = \partial^i \Psi_k - \Gamma^{i,j}_k \Psi_j$ and $B^i = h^{ij}B_j$. We get since $B_i$ is skew-symmetric,
$$
\an{u, \lap_h v}_h = -\an{\tilde \grad_h u,\tilde \grad_h v} - \an{B^iu,\partial_i v}_h - \int_{M_h} h^{ij} \overline{\partial_j u} B_i v - \an{B^i u,B_i v}_h = -\an{\grad_h u,\grad_h v}_h
$$
where $\tilde \grad_h$ is the scalar gradient.

We define 
$$\Theta (t) = \an{\psi \phi \partial_t u , \phi \partial_t u}_h + \textrm{Re}  \an{(2\psi L-L\psi)u ,u}_h$$ where $\psi$ is a real valued function of space.

To conclude this subsection, we compute $\phi \partial_t \Theta$ and $(\phi \partial_t)^2\Theta$ when $u$ solves \eqref{diraceq}. The computation is the same as in the case of a flat metrics and is mainly based on the self-adjointness of $L$, and one gets the following

\begin{proposition}\label{prop-dertheta} Let $u$ be a solution of \eqref{diraceq}. We have that $\Theta$ satisfies
\begin{eqnarray}\label{eqontheta}
\phi \partial_t \Theta &=& \re\an{[L,\psi]u,\phi \partial_t u}_h , \\
(\phi \partial_t)^2 \Theta &=&  -\frac12\re \an{[L,[L,\psi]] u , u}_h.
\end{eqnarray}
\end{proposition}

\subsection{Commutators}

We compute explicit formulae in terms of $u$ and $h$ of $\phi \partial_t \Theta$ and $(\phi \partial_t)^2\Theta$. 

\begin{proposition}\label{prop-exptheta}
Let $u$ be a solution of \eqref{diraceq}. The explicit expressions of $\phi \partial_t \Theta$ and $(\phi \partial_t)^2 \Theta$ in terms of $u$ and $h$ are
\begin{eqnarray}\label{exptheta}
\phi \partial_t \Theta  &= &- \re \Big( \int_{\mathcal{M}_h} (\lap_h \psi) \overline u \phi \partial_t u + 2  \int_{\mathcal{M}_h} \grad_h \psi \cdot \grad_h \overline u \phi \partial_t u \Big)\\
\nonumber
(\phi \partial_t)^2 \Theta & = & \int_{\mathcal{M}_h} \Big(\frac12 (\lap_h^2 \psi) + \frac14 \grad_h \psi \cdot \grad_h \mathcal R_h \Big) |u|^2 + 2  \int_{\mathcal{M}_h} ( \overline{D_j u}  D_i u) D^2(\psi)^{ij}
\end{eqnarray}
where $D^2(\psi)^{ij}= h^{il}h^{kj} \partial_l \partial_k \psi - \Lambda^{k,ij} \partial_k \psi$, from which we deduce the virial identity
\begin{multline}\label{virialid}
-\int_{\mathcal{M}_h} \Big(\frac12 (\lap_h^2 \psi) + \frac14 \grad_h \psi \cdot \grad_h \mathcal R_h \Big) |u|^2 + 2  \int_{\mathcal{M}_h} ( \overline{D_j u}  D_i u) D^2(\psi)^{ij}\\
= -(\phi \partial_t) \re \Big( \int_{\mathcal{M}_h} (\lap_h \psi) \overline u \phi \partial_t u + 2  \int_{\mathcal{M}_h} \grad_h \psi \cdot \grad_h \overline u \phi \partial_t u \Big).
\end{multline}
\end{proposition}

\begin{proof} Before we start, let us mention that $\lap_h$ and $\grad_h$ denote respectively the covariant Laplace-Beltrami operator and the gradient applied to a spinor. Therefore, we denote indifferently $\lap_h \psi$ for the scalar Laplace-Beltrami operator applied to $\psi$, and $\lap_h u$ the Laplace-Beltrami operator for Dirac bispinors applied to $u$. We also note that covariant derivatives sastisfy the Leibniz rule, as in
$$
D_j(\psi u) = \partial_j \psi u + \psi D_j u.
$$

First, the commutator between $L = \frac14\mathcal R_h +m^2- \lap_h$ and $\psi$ is given by
$$
[L, \psi] = [- \lap_h, \psi] = -\lap_h \psi - 2 \grad_h \psi \cdot \grad_h
$$
where $ \grad_h \psi \cdot \grad_h$ is the operator given by
$$
\grad_h \psi \cdot \grad_h u = h^{ij} \partial_i \psi D_j u.
$$
We deduce from that
$$
\phi \partial_t \Theta  = - \re \Big( \int_{\mathcal{M}_h} (\lap_h \psi) \overline u \phi \partial_t u + 2  \int_{\mathcal{M}_h} \grad_h \psi \cdot \grad_h \overline u \phi \partial_t u \Big).
$$

We now have that, since $m^2$ commutes with everything,
$$
[L,[L,\psi]] = [L, -\lap_h \psi] + [\frac14 \mathcal R_h , - 2 \grad_h \psi \cdot \grad_h]+ [\lap_h ,  2 \grad_h \psi \cdot \grad_h] = M_1 + M_2 + M_3.
$$

We have, in terms of distributions
$$
M_1 = (\lap_h)^2 \psi +2 \grad_h (\lap_h \psi) \cdot \grad_h \; ,\; M_2 =  \frac12 \grad_h \psi \cdot \grad_h \mathcal R_h 
$$
and hence 
$$
\an{M_1 u,u}_h = \int_{\mathcal{M}_h} (\lap_h )^2 \psi |u|^2 + \int_{\mathcal{M}_h} 2 \grad_h (\lap_h \psi) \cdot \grad_h \overline u u 
$$
and
$$
\an{M_2 u,u}_h = +\frac12 \int_{\mathcal{M}_h} \grad_h \psi \cdot \grad_h \mathcal R_h |u|^2.
$$

The term with $M_2$ is dealt with. We now deal with the term with $M_1$ that does not contain the bi-Laplace-Beltrami operator applied to $\psi$.

First, we compute $(2 \grad_h \varphi \cdot \grad_h)^*$, the symmetric of $2 \grad_h \varphi \cdot \grad_h$ for any $\varphi \in \mathcal C^0$ bounded. We have for any test functions $v,w \in \mathcal H^{10}$, by definition,
$$
\an{2 \grad_h \varphi \cdot \grad_h v, w}  =  2\int \sqrt{\det(h)} h^{ij} \partial_i \varphi D_j \overline v w d^3 x
$$
and since $D_j$ is skew-symmetric, 
$$
\an{2 \grad_h \varphi \cdot \grad_h v, w} =  2\int D_j \Big(\sqrt{\det(h)} h^{ij} \partial_i \varphi w \Big) \overline v  d^3 x
$$
and by the Leibniz rule,
$$
\an{2 \grad_h \varphi \cdot \grad_h v, w} = - 2\int \partial_j (\sqrt{\det(h)} h^{ij} \partial_i \varphi ) \overline v w d^3 x - 2 \int \sqrt{\det(h)} h^{ij} \partial_i \varphi \overline v D_j w 
$$
and finally, by the definition of $\lap_h$
$$
\an{2 \grad_h \varphi \cdot \grad_h v, w} =  - 2 \an{ v, (\lap_h \varphi) w }_h - 2 \an{v, \grad_h \varphi \cdot \grad_h w}_h
$$
in other words, $(2 \grad_h \varphi \cdot \grad_h)^* = -2 \lap_h \varphi - 2 \grad_h \varphi \cdot \grad_h$.

Note that, when we say symmetric, we always mean for the scalar product $\an{\cdot,\cdot}_h$.

This gives in particular, for $\varphi = \lap_h \psi$,
$$
\an{2 \grad_h (\lap_h \psi) \cdot \grad_h u, u} = - 2 \int_{\mathcal{M}_h} (\lap_h^2 \psi) |u|^2 -  \an{u,2 \grad_h (\lap_h \psi) \cdot \grad_h u}_h
$$
and thus
\begin{equation}\label{prevcomp}
\re \an{2 \grad_h (\lap_h \psi) \cdot \grad_h u, u} = -  \int_{\mathcal{M}_h} (\lap_h^2 \psi) |u|^2 
\end{equation}
which yields 
$$
\re \an{M_1 u,u}_h = 0.
$$
We deal with the $M_3$ term by directly taking the inner product. 
We have $M_3 = \lap_h (2 \grad_h \psi \cdot \grad_h)-(2 \grad_h \psi \cdot \grad_h)\lap_h $ and hence
$$
\an{M_3 u, u}_h = \an{\lap_h (2 \grad_h \psi \cdot \grad_h) u, u}_h - \an{(2 \grad_h \psi \cdot \grad_h)\lap_h u , u}_h
$$
and then, given the symmetry of $2 \grad_h \psi \cdot \grad_h$ (and the self-adjointness of the Laplace-Beltrami operator), 
$$
\an{(2 \grad_h \psi \cdot \grad_h)\lap_h u , u}_h = -\an{\lap_h u , 2 (\lap_h \psi) u }_h - \an{u,\lap_h (2 \grad_h \psi \cdot \grad_h) u}_h.
$$
Therefore, we get
$$
\re \an{M_3 u,u}_h = 2 \re \an{\lap_h (2 \grad_h \psi \cdot \grad_h) u, u}_h + 2 \re \an{\lap_h u , (\lap_h \psi) u }_h.
$$
The second term is given by
$$
\an{\lap_h u, \lap_h \psi u}_h = - \int_{\mathcal{M}_h} \grad_h u \cdot \grad_h(\lap_h \psi u)
$$
which decomposes, by the Leibniz rule, into
$$
\an{\lap_h u, \lap_h \psi u}_h = - \an{\grad_h(\lap_h \psi ) \cdot \grad_h u, u} - \int_{\mathcal{M}_h} \lap_h \psi \grad_h u \cdot \grad_h u
$$
and thanks to previous computations, \eqref{prevcomp}, we get
$$
\re \an{\lap_h u, \lap_h \psi u}_h = - \int_{\mathcal{M}_h} \lap_h \psi \grad_h u \cdot \grad_h u + \frac12\int_{\mathcal{M}_h} (\lap_h^2 \psi) |u|^2 .
$$
By summing up, we get
\begin{multline*}
\re \an{[L,[L,\psi]]u,u} =  \int_{\mathcal{M}_h} \Big( (\lap_h^2 \psi) + \frac12 \grad_h \psi \cdot \grad_h \mathcal R_h \Big) |u|^2  \\
-2 \int_{\mathcal{M}_h} (\lap_h \psi) \grad_h u \cdot \grad_h u + 2 \re \an{\lap_h (2 \grad_h \psi \cdot \grad_h) u, u}_h.
\end{multline*}
It remains to compute $\re I$, where
$$
 I = 2  \an{\lap_h (2 \grad_h \psi \cdot \grad_h) u, u}_h.
$$

Since for any test functions, $v,w$, we have $\an{\lap_h v,w}_h = -\an{\grad_h v,\grad_h w}$, we get
$$
I = -4 \int \sqrt{\det(h)} h^{ij} \overline{D_j \Big( h^{kl} \partial_k \psi  D_l  u\Big)}  D_i u d^3 x
$$
which decomposes, by the Leibniz rule, into
$$
I = -4 \int \sqrt{\det(h)} h^{ij} \partial_j ( h^{kl} \partial_k \psi)\overline{ D_l  u} D_i u d^3 x - 4 \int \sqrt{\det(h)} h^{ij} h^{kl} \partial_k \psi\overline{ D_l D_j  u} D_i u d^3 x.
$$
Let $II$ be the second term of the right hand side, that is, 
$$
II = - 4 \int\sqrt{\det(h)} h^{ij} h^{kl} \partial_k \psi \overline{D_l D_j  u} D_i u d^3 x.
$$
By integration by parts, we have
$$
II = 4 \int D_l \Big( \sqrt{\det(h)} h^{ij} h^{kl} \partial_k \psi D_i u \Big)  \overline{D_j u} d^3 x
$$
which decomposes into
$$
II = 4 \int \partial_l \Big( \sqrt{\det(h)} h^{ij} h^{kl} \partial_k \psi \Big) D_i u  \overline{D_j u} d^3 x -  \overline{II}.
$$
Hence,
$$
\re II = 2\int \partial_l \Big( \sqrt{\det(h)} h^{ij} h^{kl} \partial_k \psi \Big)  D_i u  \overline{D_j u} d^3 x .
$$
Therefore, by summing up
\begin{multline*}
\re I =  -4 \re \int\sqrt{\det(h)} h^{ij} \partial_j ( h^{kl} \partial_k \psi) \overline{D_l u} D_i u d^3 x 
\\
+ 2\int \partial_l \Big( \sqrt{\det(h)} h^{ij} h^{kl} \partial_k \psi \Big) D_i u   \overline{D_j u} d^3 x
\end{multline*}
that is
$$
\re I = \re \int_{\mathcal{M}_h} ( \overline{D_j u}  D_i u) D(\psi)^{ij}
$$
with 
$$
D(\psi)^{ij} = \frac2{\sqrt{\det(h)}} \partial_l \Big(\sqrt{\det(h)} h^{ij} h^{kl} \partial_k \psi \Big) - 4 h^{il} \partial_l ( h^{kj} \partial_k \psi).
$$
Since
$$
 \partial_l \Big(\sqrt{\det(h)} h^{ij} h^{kl} \partial_k \psi \Big) =  h^{ij}\partial_l \Big(\sqrt{\det(h)}  h^{kl} \partial_k \psi \Big) +  (\partial_l h^{ij}) \Big(\sqrt{\det(h)}  h^{kl} \partial_k \psi \Big),
 $$
we have 
$$
D(\psi)^{ij} = 2h^{ij} \lap_h \psi + 2h^{kl}\partial_k \psi \partial_l h^{ij} - 4 h^{il} \partial_l (h^{kj} \partial_k \psi).
$$

To sum up, we get
\begin{multline*}
(\phi \partial_t)^2 \Theta =  -\int_{\mathcal{M}_h} \Big(\frac12 (\lap_h^2 \psi) + \frac14 \grad_h \psi \cdot \grad_h \mathcal R_h \Big) |u|^2 + \int_{\mathcal{M}_h} (\lap_h \psi) \grad_h u \cdot \grad_h u \\
 - \frac12 \re \int_{\mathcal{M}_h} ( \overline{D_j u}  D_i u) D(\psi)^{ij}.
\end{multline*}

We get
$$
(\phi \partial_t)^2 \Theta = -\int_{\mathcal{M}_h} \Big(\frac12 (\lap_h^2 \psi) + \frac14 \grad_h \psi \cdot \grad_h \mathcal R_h \Big) |u|^2 + \frac12 \re \int_{\mathcal{M}_h} (\partial_j \overline u  \partial_i u) D_1(\psi)^{ij}
$$
with 
$$
D^1(\psi)^{ij} = 2 \lap_h \psi h^{ij} - D(\psi)^{ij}
$$
that is, by definition,
$$
D^1(\psi)^{ij} = -2h^{kl}\partial_k \psi \partial_l h^{ij} +4h^{il}\partial_l (h^{kj}\partial_k \psi)
$$
and by the Leibniz rule :
$$
D^1(\psi)^{ij}= 4 h^{il} h^{kj} \partial_l \partial_k \psi + 2 \partial_k \psi (- h^{kl} \partial_l h^{ij}  + 2 h^{il} \partial_l h^{kj}).
$$

Thanks to the real part, we have a symmetry in $i$ and $j$. Indeed, 
$$
\re \Big( \overline{D_j u}D_i u D_1(\psi)^{ij}\Big) = \re \Big( \overline{D_i u} D_j u D_1(\psi)^{ij}\Big).
$$
Thus, we can replace $D^1(\psi)^{ij}$ by $\frac12 \Big( D^1(\psi)^{ij} + D^1(\psi)^{ji}\Big)$, which yields
$$
\frac12 \re \int_{\mathcal{M}_h} ( \overline{D_j u}  D_i u) D^1(\psi)^{ij} = 2  \int_{\mathcal{M}_h} ( \overline{D_j u}  D_i u) D^2(\psi)^{ij}
$$
with 
$$
D^2(\psi)^{ij} =   h^{il} h^{kj} \partial_l \partial_k \psi + \frac12\partial_k \psi (- h^{kl} \partial_l h^{ij}  +  h^{il} \partial_l h^{kj} + h^{jl} \partial_l h^{ki}).
$$
We recognize the affine connection or Christoffel symbol
$$
 \Lambda^{k,ij} = h^{il}h^{jm} \Lambda^k_{lm} =\frac12 \Big( h^{kl} \partial_l h^{ij}  -  h^{il} \partial_l h^{kj} - h^{jl} \partial_l h^{ki} \Big),
$$
which yields
\begin{equation}\label{2diffform}
D^2(\psi)^{ij} =   h^{il} h^{kj} \partial_l \partial_k \psi -\Lambda^{k,ij} \partial_k \psi.
\end{equation}

Finally, we get the virial identity,
\begin{multline*}
-\int_{\mathcal{M}_h} \Big(\frac12 (\lap_h^2 \psi) + \frac14 \grad_h \psi \cdot \grad_h \mathcal R_h \Big) |u|^2 + 2  \int_{\mathcal{M}_h} (\overline{D_j u}  D_i u) D^2(\psi)^{ij}\\
= -(\phi \partial_t) \re \Big( \int_{\mathcal{M}_h} (\lap_h \psi) \overline u \phi \partial_t u + 2  \int_{\mathcal{M}_h} \grad_h \psi \cdot \grad_h \overline u \phi \partial_t u \Big).
\end{multline*}

\end{proof}

\section{The asymptotically flat case}\label{locsmot1}

The proof of Theorem \eqref{locsmoothteo} is fairly classical in this setting (see e.g. \cite{boussdanfan, cacdanluc}); nevertheless, before getting into details, let us give a brief sketch of it in order to make the various steps easier to be followed. The idea is to rely on virial identity \eqref{virialid}, plug in it a proper choice of the multiplier (with a fixed $R>0$) that we will define in subsection \ref{mult}, integrate in time and carefully estimate all the terms. By multiplier we mean the function $\psi$. We will start from the Right Hand Side: making use of a modified Hardy inequality, that will be proved in the next subsection, will allow us to estimate from above with some energy-type terms at some fixed times $0$ and $T$. Then, we will have to bound the Left Hand Side from below, which will be significantly more involved. Here we will make heavy use of our asymptotic-flatness (and smallness) assumptions to prove estimates of the different terms and, roughly speaking, treat the non-flat ones as perturbations. To absorb them,
 it will be necessary to take the 
$\sup$ in $R>0$: this will prevent us from exchanging the time and space norms in the Left Hand Side of \eqref{locsmoothest}. Eventually, we will take the sup in time and use conservation of energy.

\subsection{Useful inequalities}
We start by proving some Hardy-type and weighted estimates that will be needed in the proof of Theorem \ref{locsmoothteo}, i.e. in the asymptotically flat case. 

\begin{proposition}\label{hardyprop}
Let $m\geq0$ and assume that $h$ satisfies assumptions \eqref{eq:assaabove}, \eqref{assasy} and \eqref{assdera} with the constants $C_I$ and $C_h$ sufficiently small.
 Then for any $u$ such that $Hu\in L^2(\mathcal{M}_h)$ the following inequality holds
\begin{equation}\label{hardy}
m^2\int_{\mathcal{M}_h}|u|^2+\left[\frac{\nu^4}4-K_h\right]\int_{\mathcal{M}_h}\frac{|u|^2}{|x|^2}\leq \int_{\mathcal{M}_h}|Hu|^2
\end{equation}
for some small constant $K_h$ depending on $C_I$, $C_h$.
\end{proposition}

\begin{proof} We assumed that $g$ was complete, such that the Dirac operator $\gamma^0 \mathcal D$ is self-adjoint (see \cite{selfa}). Since $\gamma^0\mathcal D = i\phi \partial_t - H - \gamma^0 m$, we get that $H$ is essentially self-adjoint on $C^\infty_0$.

We write, as the operator $H$ is self-adjoint with respect to the inner product defined by $h$, 
\begin{equation}\label{fornext}
\int_{\mathcal{M}_h}|Hu|^2=\int_{\mathcal{M}_h}(m^2-\Delta_h) u+\frac14\int_{\mathcal{M}_h}\mathcal R_h|u|^2=I+II.
\end{equation}
Notice now that
\begin{equation}\label{I}
I=\int_{\mathcal{M}_h}\grad_hu\cdot\grad_hu + m^2\int_{\mathcal{M}_h} |u|^2.
\end{equation}
As, we recall, the operator $\grad_h u \cdot \grad_h$ denotes $h^{ij} \overline{D_i f} D_j $ where $D_i$ is the covariant derivative for spinors widely discussed in section \ref{curv} we can write
$$
\int_{\mathcal{M}_h}\grad_hu\cdot\grad_hu=\int_{\mathcal{M}_h}h^{ij}\overline{\partial_iu}\partial_ju +\int_{\mathcal{M}_h}h^{ij}\overline{B_iu} B_j u-
\int_{\mathcal{M}_h}h^{ij}(\partial_i B^j)|u|^2=I_1+I_2-I_3,
$$
where we are denoting for brevity with $B$ the field that defines the covariant derivative $D_j$ (recall \eqref{defcov} and \eqref{covvier}). Now, we rely on the asymptotically flat structure of the metrics $h$ to estimate the single terms: from \eqref{eq:assaabove} we directly have
$$
\int_{\mathcal{M}_h}h^{ij}\overline{\partial_iu}\partial_ju\geq \nu^{5/2}\int_{\mathbb{R}^3}|\grad u|^2
$$
which, by the application of standard Hardy's inequality
$$
\int_{\mathbb{R}^3}\frac{|u|^2}{|x|^2}\leq 4\int_{\mathbb{R}^3} |\grad u|^2
$$
gives
$$
I_1\geq \frac{\nu^4}4\int_{\mathcal{M}_h}\frac{|u|^2}{|x|^2}.
$$
This is to be thought of, somehow, as the "leading" term, the other ones being small perturbations of it.
The term $I_2$ is strictly positive, and therefore we can neglect it; to control $I_3$ we need instead to give an estimate for the term $\partial_i B^j$. We have shown that $B_j=\alpha_j^{ab} = f_i^a \partial_j f^{ib} + f_i^a \Lambda_{jk}^i f^{kb}$, where $f^a_b$ denotes the dreibein that connects the metrics $h$ with the flat metrics and $\Lambda_{jk}^i$ are the Christoffel symbols. Since the dreibein are constructed such that they satisfy the relation $h^{ij} = f_a^i\delta^{ab}f_b^j$, it is clear that one can bound, in the sense of the matrices, the square of $f$ with $h$; therefore, estimates \eqref{eq:assaabove}, \eqref{assasy} and \eqref{assdera} hold for the matrices $f$ with the constant $\sigma$ replaced by $\sigma/2$. After differentiating and some computations one then gets the  estimate
$$
I_3=\int_{\mathcal{M}_h}h^{ij}(\partial_i B^j)|u|^2\leq  C_1(C_h,C_I,\nu,N)\int_{\mathcal{M}_h}
\frac{|u|^2}{\langle x\rangle^{2+\tilde{\sigma}}}
$$
for some $\tilde{\sigma}>0$, and thus
$$
I_3\leq  C_1(C_h,C_I,\nu,N)\int_{\mathcal{M}_h}
\frac{|u|^2}{|x|^2}.
$$
For the term II, in an analogous fashion, we can estimate each term in the curvature by using assumptions \eqref{eq:assaabove}, \eqref{assasy} and \eqref{assdera} to eventually obtain
\begin{equation}\label{II}
II\leq C_2(C_h,C_I,\nu,N)\int_{\mathcal{M}_h}\frac{|u|^2}{|x|^2}.
\end{equation}
We should stress indeed the fact that \eqref{assdera} holds for the matrix $h$ as well, due to the well known relation $\partial_{x_j} h^{inv}=-h^{inv}(\partial_{x_j} h) h^{inv}$, with the modified constant $\nu^{-2}C_h$, and thus
$$
|h^{inv}(\partial h^{inv})\partial h|\leq \frac{N}{\nu^2}\frac{C_h^2}{\langle x\rangle^{2+2\sigma}},
$$
$$
|h^{inv}(\partial\partial h)|\leq \frac{N^2}{\nu^4}\frac{C_h}{|x|\langle x\rangle^{1+1\sigma}},
$$
$$
|h^{inv}(h^{inv}\partial h)^2|\leq \frac{N^4}{\nu^4}\frac{C_h}{\langle x\rangle^{2+2\sigma}}.
$$

Putting all together, we thus have
$$
\int_{\mathcal{M}_h}|Hu|^2\geq \left[\frac{\nu^4}4-C_1-C_2\right]\int_{\mathcal{M}_h}\frac{|u|^2}{|x|^2}+m^2\int_{\mathcal{M}_h}|u|^2.
$$
Notice that, even though we did not write it explicitly in order to keep the presentation as light as possible, it is possible to take the constants $C_1$ and $C_2$ in order to have $K_h:=C_1+C_2<\frac{\nu^4}4$ provided the constants $C_h$ and $C_I$ are small enough, and this concludes the proof.
%
\end{proof}

\begin{remark}
In what follows we will make also use of the following estimate, which holds for any $\varepsilon\in (0,1)$,
\begin{equation}\label{hardy2}
m^2\int_{\mathcal{M}_h}|u|^2+\left[\frac{(1-\varepsilon)\nu^4}4-K_h\right]\int_{\mathcal{M}_h}\frac{|u|^2}{|x|^2}+\varepsilon\int_{\mathcal{M}_h} |\grad u|^2\leq \int_{\mathcal{M}_h}|H u|^2
\end{equation}
that can be obtained by combining \eqref{hardy} with the obvious inequality
$$
\varepsilon\int_{\mathcal{M}_h} |\grad u|^2+\frac{(1-\varepsilon)}4\nu^4\int_{\mathcal{M}_h}\frac{|u|^2}{|x|^2}\leq\int_{\mathcal{M}_h}|\grad u|^2.
$$
\end{remark}
%

In the following Proposition we collect a number of other weighted inequalities that will be needed in the proofs of the main results.

\begin{proposition}
For any $\sigma\in (0,1)$ and any $u\in C^\infty_0(\mathbb{R}^n)$ the following estimates hold
\begin{equation}\label{prev1}
\int_{\mathcal{M}_h}\frac{|f|^2}{\langle x\rangle^{1+\sigma}}\leq 8\sigma^{-1}C_{N,\nu}\|u\|_{Y}^2,
\end{equation}
\begin{equation}\label{prev2}
\sup_{R>1}\int_{\mathcal{M}_h\cap B_R^c}\frac{R^2}{|x|^5}|u|^2\leq C_{N,\nu} \|u\|_X^2,
\end{equation}
\begin{equation}\label{prev3}
\int_{\mathcal{M}_h\cap B_1^c}\frac{|u|^2}{|x|^2\langle x\rangle^{1+\sigma}}\leq 2\sigma^{-1}C_{N,\nu}\|u\|_X^2.
\end{equation}
\begin{equation}\label{prev4}
\|u\|_{X}^2\leq C_{N,\nu} \left[4\sup_{R>1}\frac1{R^2}\int_{\mathcal{M}_h\cap S_R} |u|^2+13\|\grad u\|_Y^2\right]
\end{equation}
\begin{equation}\label{prev5}
\| u\|_{Y}^2\leq 3C_{N,\nu} \left(2\|\grad u\|_Y^2+\|u\|_{X}^2\right)
\end{equation}
where the constant $\displaystyle C_{N,\nu}= \left(\frac{\nu}N\right)^{3/2}$.
\end{proposition}

\begin{proof}
For the proof we refer to \cite{cacdanluc} section 3: the generalization from the Euclidean case to our perturbative setting is straightforward under assumptions \eqref{eq:assaabove}-\eqref{detest}.
\end{proof}


\subsection{Choice of the multiplier}\label{mult} We here present the multiplier function $\psi$ that will be used in the main proof. We define the radial function $\psi(x)$ as
$$
\psi_0(x)=\int_0^{r}\psi_0'(s)ds
$$
where
$$
\psi_0'(r)=\begin{cases}
\displaystyle
\frac{r}3,\quad r\leq1\\
\displaystyle
\frac12-\frac1{6r^2},\quad r>1
\end{cases}
$$
(with a slight abuse we are using the same notation for $\psi(x)$ and $\psi(r)$ where $r=|x|$).
We then define the scaled function
$$
\psi_R(r):= R\psi_0\left(\frac{r}{R}\right),\qquad R>0
$$
for which we have
$$
\psi_R'(r)=\begin{cases}
\displaystyle
\frac{r}{3R},\quad r\leq R\\
\displaystyle
\frac12-\frac{R^2}{6r^2},\quad r>R.
\end{cases}
$$
Moreover, 
$$
\psi_R''(r)=\begin{cases}
\displaystyle
\frac{1}{3R},\quad r\leq R\\
\displaystyle
\frac{R^2}{3r^3},\quad r>R.
\end{cases}
$$
$$
\psi_R'''(r)=-\frac{R^2}{r^4}{\bf 1}_{r\geq R},
$$
$$
\psi_R^{iv}=4\frac{R^2}{r^5}{\bf 1}_{r\geq R}-\frac1{R^2}\delta(r-R)
$$
where ${\bf 1}_{r\geq R}$ denotes the characteristic function of the set $\{x:|x|\geq R\}$ and $\delta$ the standard Dirac delta distribution.
Notice that, for every $r\geq0$,
\begin{equation*}
\psi_R''(r)\leq \frac{1}{2\max\{R,r\}}\leq \frac1{2r},\qquad \psi_R'(r)\leq \frac12.
\end{equation*}
Moreover, notice that
\begin{equation}\label{dirder}
\psi_R''-\frac{\psi_R'}{r}=\begin{cases}
\displaystyle
0,\quad r\leq R\\
\displaystyle
-\frac{1}{2r^2}\left(1-\frac{R^2}{r^2}\right),\quad r>R.
\end{cases}
\end{equation}
In the following we shall simply denote with $\psi=\psi_R$ for a fixed $R$. 

\begin{remark}
Our choice of the multiplier $\psi$ here is classical, and already highly used in several papers to prove smoothing estimates for different dispersive equations, also in some perturbative settings. The function $\psi$ is a mix of the so called \emph{virial} and \emph{Morawetz} multipliers, which are respectively given by $|x|$ and $|x|^2$; originally, the choice of such a function was dictated by the conditions of having a negative bi-Laplacian and a positive Hessian, together with some good decay at infinity. 
Of course, in a fully variable coefficients setting, this properties are much more difficult to be fulfilled, and a smart choice of the multiplier, to the best of our knowledge, has never been attempted in this general case. Therefore our choice is motivated by perturbative arguments: the idea of the proof will be that the "leading" terms in the inequality will mainly recover the ones in the flat case, while the terms  involving the variable coefficients will be treated by "smallness" arguments.
\end{remark}

\subsection{Estimate of the right hand-side (RHS)}
We use the Dirac equation \eqref{diraceqcurv} to rewrite the right hand side of \eqref{virialid} as (notice that the mass term vanishes when taking the real part)
\begin{equation}\label{rhss}
(\phi\partial_t){\rm Re} \left(i\int_{\mathcal{M}_h}(\Delta_h\psi) Hu \: \overline{u}+2\int_{\mathcal{M}_h}\grad_h\psi\cdot \grad_h\overline{u}\:Hu\right).
\end{equation}
First of all, by the application of Young's inequality we can write the estimate
\begin{equation}\label{rhsest}
\left|\int_{\mathcal{M}_h}\big[(\Delta_h\psi) Hu \: \overline{u}+2\grad_h\psi\cdot \grad_h\overline{u}\:Hu\big]\right|
\end{equation}
\begin{equation*}
\leq
\frac32\|Hu\|_{L^2(\mathcal{M}_h)}^2+\|\grad_h\psi\cdot \grad_h\overline{u}\|_{L^2(\mathcal{M}_h)}^2+\frac12\|\Delta_h\psi\:\overline{u}\|_{L^2(\mathcal{M}_h)}^2.
\end{equation*}
Recalling \eqref{lappsi}, \eqref{detest} and \eqref{hardy} we then have
\begin{eqnarray}\label{rhs1}
\|\Delta_h\psi\:\overline{u}\|_{L^2(\mathcal{M}_h)}^2&\leq& \frac{3N}4\left\|x|^{-1}u\right\|_{L^2(\mathcal{M}_h)}^2+\frac{C_\grad}{4\nu^{3/2}}\left\|\langle x\rangle^{-1-\alpha}u\right\|_{L^2(\mathcal{M}_h)}^2
\\
\nonumber
&\leq& C\left(\frac{3N}4+\frac{C_\grad N^{3/2}}{4}\right)\|Hu\|_{L^2(\mathcal{M}_h)}
\end{eqnarray}
where the constant $C$ above will depened on $C_I$ and $C_h$.
Moreover, due to condition \eqref{assdera}, we have
$$
|\grad_h\psi\cdot\grad_h \overline{u}|\leq N/2|\grad u|,
$$
so that applying estimate \eqref{hardy2} we obtain
\begin{equation}\label{rhs2}
\|\grad_h\psi\cdot \grad_h\overline{u}\|_{L^2(\mathcal{M}_h)}^2\leq C\|Hu\|_{L^2(\mathcal{M}_h)}^2
\end{equation}
with a constant $C$ depending on $N$, $\nu$, $C_I$ and $C_h$.
Therefore, multiplying \eqref{rhs1} times $\phi^{-1}$ and integrating in time between $0$ and $T$ we obtain,
plugging \eqref{rhs1} and \eqref{rhs2} into \eqref{rhsest},
\begin{equation}\label{rhsfinal}
\left|\int_0^T \phi^{-1}(t)(\phi(t)\partial_t)\int_{\mathcal{M}_h}\big[(\Delta_h\psi) Hu \: \overline{u}+2\grad_h\psi\cdot \grad_h\overline{u}\:Hu\big]\right|
\end{equation}
\begin{equation*}
\lesssim \|Hu(T)\|_{L^2(\mathcal{M}_h)}^2+\|Hu(0)\|_{L^2(\mathcal{M}_h)}^2.
\end{equation*}

\subsection{Estimate of the left hand-side (LHS)}
We now deal with the left hand side of identity \eqref{virialid}. We estimate each term separately, and start with the one involving the gradient, namely
\begin{equation}\label{firsterm}
2  \int_{\mathcal{M}_h} ( \overline{D_j u}  D_i u) D^2(\psi)^{ij}.
\end{equation}
Recalling \eqref{2diffform}, we treat separately terms involving derivatives on the coefficients from the others. Concerning $\Lambda^{k,ij}$ we have 
$$
|(\overline{D_j u}  D_i u)\Lambda^{k,ij}|\leq 3 |h_{inv}| |h_{inv}'| |\grad_h u|^2|\psi'|
$$
and thus, by our assumptions \eqref{assdera} and from the bound on $\psi'$, 
\begin{equation}\label{RRR}
|(\overline{D_j u}  D_i u)\Lambda^{k,ij}|\leq \frac32 N C_h\langle x\rangle^{-1-\sigma}|\grad_h u|^2.
\end{equation}
Turning to the other term, we use the fact that $\psi$ is radial to rewrite it as follows
$$
h^{il} h^{kj} \partial_l \partial_k \psi=h^{il} h^{kj} \hat{x}_l\hat{x}_k\left(\psi''-\frac{\psi'}{|x|}\right)
+h^{il}h^{jl}\frac{\psi'}{|x|}.
$$
We restrict the quantity above first in the region $|x|\leq R$ where, notice, $\psi''=\frac{\psi'}{|x|}$. Therefore,
\begin{eqnarray}\label{inR}
\nonumber
{\bf 1}_{|x|\leq R}(\overline{D_j u}  D_i u)h^{il} h^{kj} \partial_l \partial_k \psi&=&
\frac1{3R}{\bf 1}_{|x|\leq R} h^{il}h^{jl}(\overline{D_j u}  D_i u)
\\
&\geq&
\frac{\nu^2}{3R}{\bf 1}_{|x|\leq R}|\grad_h u|^2
\end{eqnarray}
where in the last inequality we have used \eqref{eq:assaabove}. In the region $|x|> R$ we have instead
\begin{equation}\label{outR}
{\bf 1}_{|x|> R}h^{il} h^{kj} \partial_l \partial_k \psi=
\frac1{2|x|}\left[h^{il}h^{jl}-h^{il}h^{kj}\hat{x}_l\hat{x}_k\right]+\frac{R^2}{2|x|^3}\hat{x}_l\hat{x}_kh^{ik}h^{jl}-\frac{R^2}{6|x|^3}h^{il}h^{jl}\geq0
\end{equation}
in the sense of matrices (notice that $h^{il}h^{jl}-h^{il}h^{kj}\hat{x}_l\hat{x}_k\geq 0$ in the sense of matrices). We can therefore neglect this term.

We thus multiply \eqref{firsterm} by $\phi^{-1}$, and integrate in time between $0$ and $T$. Exchanging the integrals and applying \eqref{RRR}, \eqref{inR} and \eqref{outR} therefore gives (recall \eqref{prev1})
\begin{equation}\label{termgrad}
2  \int_{\mathcal{M}_h}\phi^{-1} \int_0^T(\partial_j \overline u  \partial_i u) D^2\psi^{ij}
\geq
\frac{2\nu^2}{3R}\int_{\mathcal{M}_h\cap B_R}\|\grad u\|_{L^2_{\phi,T}}^2-C_{D^2}\|\grad u\|^2_{YL^2_{\phi,T}}
\end{equation}
with the constant
\begin{equation}\label{C}
C_{D^2}=\frac{12\nu^{3/2}C_h}{\sqrt{N}\sigma}
\end{equation}
Now we turn to the bi-Laplacian term, that is
\begin{equation}\label{bilterm}
\frac12\int_{\mathcal{M}_h} (\lap_h^2 \psi)|u|^2.
\end{equation}
First of all observe that
$$
\Delta_h(fg)=(\Delta_hf)g+2\grad_hf\cdot\grad_hg+(\Delta_hg)f,
$$
so that we can write, after some manipulations
$$
\Delta_h^2\psi=\Delta_h\Delta_h\psi=I+II+III+IV
$$
with
$$
I=\hat{h}\cdot \Delta_h\psi''+(\overline{h}-\hat{h})\Delta\left(\frac{\psi'}{|x|}\right),
$$
$$
II=A\hat{h}\cdot \psi''+A(\overline{h}-\hat{h})\cdot\frac{\psi'}{|x|},
$$
$$
III=2\grad_h\hat{h}\cdot \grad_h\psi''+2\grad_h(\overline{h}-\hat{h})\cdot\grad_h\frac{\psi'}{|x|},
$$
$$
IV=\Delta_h\left(\frac{1}{\sqrt{\det(h)}}\partial_j(\tilde{h}^{jk}\hat{x}_k\psi')\right).
$$
We separate terms involving derivatives on the coefficients of $h^{jk}$ (which will be of perturbative nature) from the others. After some long winded but not difficult computations (see \cite{cacdanluc} section 4.4 for further details)
one gets
$$
\Delta_h^2\psi=S(x)+R(x)
$$
\begin{equation*}
\begin{split}
  S(x)=&
  \textstyle
  \widehat{h}^{2}{\psi^{iv}}+
  2\widehat{h}(\overline{h}-\widehat{h})
  \frac{\psi'''}{|x|}+
  \frac{(\overline{h}-\widehat{h})(\overline{h}-3 
  \widehat{h})}{|x|^{2}}
  \left(
  \psi''-\frac{\psi'}{|x|}
  \right)+
    \\
  &+
  \textstyle
  \frac{2}{|x|^{2}}
  [
  h^{\ell m}h^{\ell m}
  -\overline{h}\widehat{h}
  -4(|h \widehat{x}|^{2}-\widehat{h}^{2})
  ]
  \left(
  {\psi''}-\frac{\psi'}{|x|}
  \right)+
    \\
  &
  \textstyle
  +\frac{4}{|x|}[|h \widehat{x}|^{2}-\widehat{h}^{2}]
  \left(
  {\psi'''}
  -\frac{\psi''}{|x|}
  +\frac{\psi'}{|x|^{2}}
  \right)
\end{split}
\end{equation*}
and
\begin{equation*}
  \textstyle
\begin{split}
  \sqrt{\det(h)}R(x)=
  &
  \widehat{h}\partial_m(\tilde{h}^{\ell m})\widehat{x}_{m}{\psi'''}+
  (\overline{h}-\widehat{h})
 \partial_k(\tilde{h}^{jk})\widehat{x}_{k}
  \textstyle
  \left(
  \frac{\psi''}{|x|}-\frac{\psi'}{|x|^{2}}
  \right)+
    \\
  &+
    [
  \partial_{j}(\tilde{h}^{jk}\partial_k (h^{\ell m})\widehat{x}_{\ell}
  \widehat{x}_{m})
  +\partial_{j}(\tilde{h}^{jk}h^{\ell m})\partial_{k}(\widehat{x}_{\ell}
      \widehat{x}_{m})
  ]
  \textstyle
  \left(
  {\psi''}-\frac{\psi'}{|x|}
  \right)
     \\
  & +
  \sqrt{\det(h)}(\Delta_h\overline{h})\frac{\psi'}{|x|}
+
    2\sqrt{\det(h)}h^{jk}\partial_kh^{\ell m}\widehat{x}_{\ell}\widehat{x}_{m}
    \widehat{x}_{j}
    \textstyle
    \left({\psi'''}-\frac{\psi''}{|x|}\right)
    +    \\ 
  &+
  \textstyle
  2\sqrt{\det(h)}h(\grad \overline{h},\grad \frac{\psi'}{|x|})+
\sqrt{\det(h)}\Delta_h\left(\frac{1}{\sqrt{\det(h)}}\partial_j(\tilde{h}_{jk}\hat{x}_k\psi')\right).
\end{split}
\end{equation*}
In our assumptions on the metric $h$ and noticing that by the definition of $\psi$ we have
\begin{equation*}
  \textstyle
  |\psi'|\le \frac{|x|}{2 (R\vee|x|)},\qquad
  |\psi''|\le \frac{n-1}{2n (R\vee|x|)},\qquad
  |\psi'''|\le \frac{n-1}{2 (R\vee|x|)|x|},
\end{equation*}
the remainder term $R(x)$ can be estimated as
\begin{equation}\label{rest}
|R(x)|\leq \frac{36 C_h(N+C_h)}{|x|\langle x\rangle^{1+\sigma}\max\{R,|x|\}}.
\end{equation}
We regroup the terms in $S(x)$ to write
\begin{eqnarray}
\label{Srev}
S(x)&=&\hat{h}^2\psi^{iv}+\left(2\overline{h}\hat{h}-6\hat{h}^2+4|h\hat{x}|^2\right)\frac{\psi'''}{|x|}
\\
\nonumber
&+&
\left(2h^{\ell m}h^{\ell m}+\overline{h}^2-6\overline{h}\hat{h}+15 \hat{h}^2-12|h\hat{x}|^2\right)\left(\frac{\psi''}{|x|^2}-\frac{\psi'}{|x|^3}\right).
\end{eqnarray}
Now, plugging our choice of the weight into \eqref{Srev} gives (recall \eqref{dirder}
\begin{equation*}
S(x)=-\frac1{R^2}\hat{h}^2\delta(|x|-R)
\end{equation*}
for $r\leq R$ and, after rearranging the terms,
$$
S(x)=2\left(3\hat{h}-\overline{h}\right)\hat{h}\frac{R^2}{|x|^5}-6(|h\hat{x}|^2-\hat{h}^2))\frac{R^2}{|x|^5}
$$
$$
-\left(2h^{\ell m}h^{\ell m}+\overline{h}^2-6\overline{h}\hat{h}+15 \hat{h}^2-12|h\hat{x}|^2\right)
\left(1-\left(\frac{R}{|x|}\right)^2\right)\frac1{2|x|^3}
$$
for $|x|>R$. As we can write $h_{inv}(x)=I+\varepsilon(x)$  (meaning $\varepsilon_{jk}=h^{jk}-\delta_{jk})$, we have 
$$
h^{\ell m}h^{\ell m}=\delta_{\ell m\ell m}+2\delta_{\ell m}\varepsilon_{\ell m}+\varepsilon_{\ell m}\varepsilon_{\ell m}=3+2\overline{\varepsilon}+\varepsilon_{\ell m}\varepsilon_{\ell m}
$$
as well as
$$
\hat{h}=1+\hat\varepsilon,\qquad \overline{a}+\overline{h},\qquad |h\hat{x}|^2=1+2\hat{\varepsilon}+|\varepsilon\hat{x}|^2.
$$
Notice also that by assumption \eqref{assasy} $|\varepsilon(x)|=|h_{inv}(x)-I|\leq C_I\langle x\rangle^{-\sigma}<1$ and therefore
$$
|\overline\varepsilon|\leq 3C_I\langle x\rangle^{-\sigma},\qquad |\hat{\varepsilon}|\leq C_I\langle x\rangle^{-\sigma},\qquad |\varepsilon \hat{x}|\leq C_I\langle x\rangle^{-\sigma}
$$
so that 
\begin{eqnarray}\label{appt}
\nonumber
2h^{\ell m}h^{\ell m}+\overline{h}^2-6\overline{h}\hat{h}+15 \hat{h}^2-12|h\hat{x}|^2&=&
4\overline{\varepsilon}-12\hat{\varepsilon}+2\varepsilon_{\ell m}\varepsilon_{\ell m}+\overline{\varepsilon}^2
\\
\nonumber
&&-6\overline{\varepsilon}\hat{\varepsilon}+15\hat{\varepsilon}^2-12|\varepsilon \hat{x}|^2
\\
\nonumber
&\geq&
4\overline{\varepsilon}-12\hat{\varepsilon}-6\overline{\varepsilon}\hat{\varepsilon}-12|\varepsilon \hat{x}|^2
\\
&\geq&
-46 C_I\langle x\rangle ^{-\sigma}.
\end{eqnarray}
Also, as $1-C_I\leq \hat{h}\leq 1+C_I$, we have
\begin{equation*}
-\hat{h}^2\leq -(1-C_I)^2,
\end{equation*}
and
\begin{equation}\label{apprev}
 \left(3\hat{h}-\overline{h}\right)\hat{h}\leq 6C_I(1+C_I)\leq 12C_I.
\end{equation}
Therefore, under our assumptions and with our choice of the multiplier $\psi$, we obtain the estimates
$$
S(x)\leq -(1-C_I)^2\frac1{R^2}\delta(|x|-R)\qquad {\rm for}\;|x|\leq R,
$$
and, with a bit more careful computations that essentially rely on \eqref{appt} and \eqref{apprev}
$$
S(x)\leq 24 C_I\left[\frac{R^2}{|x|^5}+\frac1{|x|^3\langle x\rangle^\sigma}\right]\qquad {\rm for}\;|x|> R.
$$
We now multiply times $\phi^{-1}$ and integrate in time \eqref{bilterm} from $0$ to $T$: this gives
$$
-\int_0^T\phi^{-1}\int_{\mathcal{M}_h} \Delta_h^2\psi|u|^2=-\int_{\mathcal{M}_h}\Delta_h^2\psi\|u\|_{L^2_{\phi,T}}^2=I+II
$$
with
$$
I=-\int_{\mathcal{M}_h}S(x)\|u\|_{L^2_{\phi,T}}^2,\qquad II=-\int_{\mathcal{M}_h}R(x)\|u\|_{L^2_{\phi,T}}^2
$$
and estimate the two terms separately. For the $S(x)$ term we get, thanks to \eqref{prev2} and \eqref{prev3},
$$
I\geq(1-C_I)^2\frac1{R^2}\int_{\mathcal{M}_h\cap S_R}\|u\|_{L^2_{\phi,T}}^2-\frac{72C_I}\sigma\left(\frac{\nu}N\right)^{3/2}\|u\|_{XL^2_{\phi,T}}^2.
$$
The $R(x)$ term can be instead estimated with
$$
II\geq -36C_h(N+C_h)\int_0^T\left[\int_{\mathcal{M}_h\cap B_R}+\int_{\mathcal{M}_h\cap B_R^c}\right]\frac{|u|^2}{\phi(t)|x|^2\langle x\rangle^{1+\sigma}}
$$
where $B_R^c$ is the complementary set of $B_R$, that is the region where $r>R$.
Thanks to \eqref{prev3} we have
\begin{equation}\label{R_2}
\int_0^T\int_{\mathcal{M}_h\cap B_R}\frac{|u|^2}{\phi(t)|x|^2\langle x\rangle^{1+\sigma}}\leq \int_{\mathcal{M}_h\cap B_R}\frac{\|u\|_{L^2_{\phi,T}}^2}{|x|^2\langle x\rangle^{1+\sigma}}\leq \frac2{\sigma}\left(\frac{\nu}N\right)^{3/2}\|u\|_{XL^2_{\phi,T}}^2
\end{equation}
and thanks to \eqref{1norm} and \eqref{hardy}
\begin{equation}\label{R_3}
\int_0^T\int_{\mathcal{M}_h\cap B_1}\frac{|u|^2}{\phi(t)|x|^2\langle x\rangle^{1+\sigma}}\leq
\int_0^T\int_{\mathcal{M}_h\cap B_1}\frac{|u|^2}{\phi(t)|x|^2}\leq 4\|\grad u\|^2_{L^2(\mathcal{M}_h\cap B_1)L^2_{\phi,T}}.
\end{equation}
From \eqref{R_2} and \eqref{R_3} we thus obtain
\begin{equation*}
II\geq -\frac{324 C_h(N+C_h)}\sigma\left(\frac{\nu}N\right)^{3/2}\left(\|u\|_{XL^2_{\phi,T}}^2+\|\grad u\|^2_{L^2(\mathcal{M}_h\cap B_1)L^2_{\phi,T}}\right).
\end{equation*}
Putting all together gives
\begin{equation*}
-\frac12\int_0^T\phi^{-1}\int_{\mathcal{M}_h} \Delta_h^2\psi|u|^2\geq
\frac{(1-C_I)^2}2\frac1{R^2}\int_{\mathcal{M}_h\cap S_R}\|u\|_{L^2_{\phi,T}}^2
\end{equation*}
\begin{equation*}
-\frac1\sigma\left(\frac{\nu}N\right)^{3/2}\left[
36C_I\|u\|_{XL^2_{\phi,T}}^2+162C_h(N+C_h)\left(\|u\|_{XL^2_{\phi,T}}^2+\|\grad u\|^2_{L^2(\mathcal{M}_h\cap B_1)L^2_{\phi,T}}\right)
\right]
\end{equation*}
Recalling \eqref{1norm} eventually gives
\begin{eqnarray}\label{finalbil}
-\frac12\int_0^T\phi(t)^{-1}\int_{\mathcal{M}_h} \Delta_h^2\psi|u|^2&\geq&
\frac{(1-C_I)^2}2\frac1{R^2}\int_{\mathcal{M}_h\cap S_R}\|u\|_{L^2_{\phi,T}}^2
\\
\nonumber
&-&C_{\Delta^2}^I\|u\|_{XL^2_{\phi,T}}^2-C_{\Delta^2}^{II}
\|\grad u\|^2_{YL^2_{\phi,T}}.
\end{eqnarray}
where the constants are explicitly given by
\begin{equation}\label{C_Delta}
C_{\Delta^2}^I=\left(\frac{\nu}N\right)^{3/2}\frac{36 C_I+162C_h(N+C_h)}\sigma,\quad
C_{\Delta^2}^{II}=\left(\frac{\nu}N\right)^{3/2}\frac{162C_h(N+C_h)}\sigma
\end{equation}
We now turn to the last term of \eqref{virialid} that is
\begin{equation}\label{Riccterm}
\int_{\mathcal{M}_h} \grad_h \psi \cdot \grad_h \mathcal R_h |u|^2.
\end{equation}
Notice that it involves only terms with derivatives on $h$ (and indeed vanishes in the flat case). Therefore, using repeatedly assumptions \eqref{assdera}, it is not difficult to show that
\begin{eqnarray}\label{crucialcurv}
\nonumber
|\grad\mathcal R_h(x)|&\leq&\frac{3C_h(1+18N^2)}{\langle x\rangle^{3+3\sigma}}
+\frac{3NC_h^2+9NC_h^2+9N^3C_h^2}{\langle x\rangle^{2+2\sigma}|x|}+\frac{3N^2C_h^3}{\langle x\rangle^{1+\sigma}|x|^2}
\\
&\leq&
\frac{C_{\mathcal{R}}}{\langle x\rangle^{1+\sigma}|x|^2}.
\end{eqnarray}
where the constant $C_{\mathcal{R}}$ is the sum of the three numerators above, that is
$$
C_{\mathcal{R}}=3C_h(1+18N^2)+3NC_h^2(4+3N^2)+3C_h^3N^2.
$$
We now multiply as usual \eqref{Riccterm} times $\phi^{-1}$ and integrate in time between $0$ and $T$: following calculations and relying on \eqref{prev3}-\eqref{hardy} yield the estimate
\begin{eqnarray}\label{ricfinal}
\nonumber
\int_0^T\phi(t)^{-1}\int_{\mathcal{M}_h} \grad_h \psi \cdot \grad_h \mathcal R_h |u|^2&\geq&
-C_{\mathcal{R}}\int_0^T\left[\int_{\mathcal{M}_h\cap B_R}+\int_{\mathcal{M}_h\cap B_R^c}\right]\frac{|u|^2}{\phi(t)|x|^2\langle x\rangle^{1+\sigma}}
\\
\nonumber
&\geq&-C_{\mathcal{R}}\left[\frac{2}\sigma\left(\frac{\nu}N\right)^{3/2}\|u\|_{XL^2_{\phi,T}}^2+2\left(\frac{\nu}N\right)^{3/2}\|\grad u\|_{YL^2_{\phi,T}}\right]
\\
&=&
-4C_{\mathcal{R}}^{I}\|u\|_{XL^2_{\phi,T}}^2-C_{\mathcal{R}}^{II}\|\grad u\|_{YL^2_{\phi,T}}
\end{eqnarray}
with
\begin{equation}\label{C_R}
C_{\mathcal{R}}^I=\frac{C_{\mathcal{R}}}{2\sigma}\left(\frac{\nu}N\right)^{3/2},\qquad C_{\mathcal{R}}^{II}=\frac{C_{\mathcal{R}}}2\left(\frac{\nu}N\right)^{3/2}.
\end{equation}

\subsection{Conclusion of the proof.}\label{conclusion}
We multiply times $\phi^{-1}$ and integrate in time identity \eqref{virialid} from $0$ to $T$, exchange integrals and use \eqref{termgrad}, \eqref{finalbil}, \eqref{ricfinal} for the left hand side and \eqref{rhsfinal} for the right hand side to obtain
\begin{equation}\label{conc1}
\frac{(1-C_I)^2}2\frac1{R^2}\int_{\mathcal{M}_h\cap S_R}\|u\|_{L^2_{\phi,T}}^2+\frac{2\nu^2}{3R}\int_{\mathcal{M}_h\cap B_R}\|\grad v\|_{L^2_{\phi,T}}^2
\end{equation}
\begin{equation*}
-(C_{\Delta^2}^I+C_{\mathcal{R}}^{I})\|u\|_{XL^2_{\phi,T}}^2-(C_{D^2}+C_{\Delta^2}^{II}+C_{\mathcal{R}}^{II})\|\grad u\|_{YL^2_{\phi,T}}
\end{equation*}
\begin{equation*}
\leq C_{\nu,N,\sigma}\|Hu(T)\|_{L^2(\mathcal{M}_h)}^2+\|Hu(0)\|_{L^2(\mathcal{M}_h)}^2
\end{equation*}
where the constants are explicit and given by \eqref{C}, \eqref{C_Delta} and \eqref{C_R}. We also stress that the constant $C=C_{\nu,N,\sigma}$ does not depend on $R$. We now take the $\sup$ over $R>1$ on the left hand side of \eqref{conc1} (notice that only the first two terms of inequality above depend on $R$). We use \eqref{prev4} to estimate, for $0<\theta<1$, 
\begin{equation}\label{concapp}
\frac{(1-C_I)^2}2\sup_{R>1}\frac1{R^2}\int_{\mathcal{M}_h\cap S_R}\|u\|_{L^2_{\phi,T}}^2
\geq (1-\theta)\frac{(1-C_I)^2}2\sup_{R>1}\frac1{R^2}\int_{\mathcal{M}_h\cap S_R}\|u\|_{L^2_{\phi,T}}^2
\end{equation}
\begin{equation*}
+
\theta(1-C_I)^2\left[\frac14\left(\frac\nu{N}\right)^{3/2}\|u\|_{XL^2_{\phi,T}}^2-\frac{13}4\|\grad u\|_{YL^2_{\phi,T}}\right].
\end{equation*}
Thanks to our assumption \eqref{assasy}, we can take $\nu=1-C_I$, such that
\begin{equation*}
\sup_{R>1}\frac{2\nu^2}{3R}\int_{\mathcal{M}_h\cap B_R}\|\grad u\|_{L^2_{\phi,T}}^2
\geq
\frac23(1-C_I)^2\|\grad u\|_{YL^2_{\phi,T}}^2.
\end{equation*}
Choosing $\theta$ in \eqref{concapp} such that $\frac{13\theta}4\leq \frac23$ (e.g. $\theta=1/5$) and using the simple property
$$
\sup_{R}(F_1(R)+F_2(R))\geq \frac12\left(\sup_RF_1(R)+\sup_RF_2(R)\right)
$$
for positive $F_1$ and $F_2$ yields
\begin{equation*}
\frac{(1-C_I)^2}2\sup_{R>1}\frac1{R^2}\int_{\mathcal{M}_h\cap S_R}\|u\|_{L^2_{\phi,T}}^2+\sup_{R>1}\frac{2\nu^2}{3R}\int_{\mathcal{M}_h\cap B_R}\|\grad u\|_{L^2_{\phi,T}}^2
\end{equation*}
\begin{equation*}
\geq
(1-C_I)^2\left(\frac1{40}\left(\frac\nu{N}\right)^{3/2}\|u\|_{XL^2_{\phi,T}}^2+\frac1{120}\|\grad u\|_{YL^2_{\phi,T}}^2\right)
\end{equation*}
which plugged into \eqref{conc1} finally gives
\begin{equation*}
M_1\|u\|_{XL^2_{\phi,T}}^2+M_2\|\grad u\|_{YL^2_{\phi,T}}^2\leq C_{\nu,N,\sigma}\|Hu(T)\|_{L^2(\mathcal{M}_h)}^2+\|Hu(0)\|_{L^2(\mathcal{M}_h)}^2
\end{equation*}
with
$$
M_1=\frac{(1-C_I)^2}{40}-C_{\Delta^2}^I-C_{\mathcal{R}}^{I}
$$
and
$$
M_2=\frac{(1-C_I)^2}{120}-C_{D^2}-C_{\Delta^2}^{II}-C_{\mathcal{R}}^{II}.
$$
The proof is concluded provided the constants $M_1$ and $M_2$ are positive,  i.e. if the constants $C_I$  and $C_h$ are small enough, by letting $T$ to infinity and using the conservation of the $L^2$-norm of $Hu$, which is standard.

\section{The warped products case}\label{locsmot2}

We dedicate this section to prove Theorem \ref{locsmoothteo2}. First of all, we notice that if $h$ is in the form \eqref{defWP} the following result holds.
%
\begin{proposition}\label{prop-computWP} Let $h$ be a warped product. We have 
\begin{equation}\label{curvature}
\mathcal R_h = - 2\frac{d''}{d} + \frac12 \Big( \frac{d'}{d}\Big)^2 + \frac1{d} \mathcal R_\kappa
\end{equation}
and
\begin{equation}\label{covder}
\Lambda^{1,ij} = 0 \textrm{ if } i=1 \textrm{ or } j= 1 \textrm{ and } \Lambda^{1,ij} = -\frac12 \frac{d'}{d^2} \kappa^{ij} \textrm{ otherwise.}
\end{equation}
\end{proposition}

\begin{proof} The proof is straightforward computation. \end{proof}

The strategy to prove Theorem \ref{locsmoothteo2} is the same we have seen in details in the previous section to deal with the asymptotically flat case, and thus consists in applying the virial identity \eqref{virialid} to an appropriate function $\psi$, and then estimate the various terms. We will deal with the three different cases separately.

Before getting into details, let us comment on the choice of the multiplier and on some of its basic properties. The multiplier we choose is very similar to the flat case or, indeed, the asymptotically flat case one. First, we take $\psi$ to be radial, which means here that it depends only on the priviledged variable $x^1 = r$. We divide $\psi$ into two sectors: one below a chosen $R$ ($r\leq R$) and one above $R$.  Below $R$, we choose the map $\psi'$ to be affine. This is important because of the integral
$$
\int_{\mathcal{M}_h} D^2(\psi)^{ij} \partial_i \overline u \partial_j u;
$$
as $\psi$ is radial, this term is equal to
$$
\int_{\mathcal{M}_h} h^{11}\psi'' |D_r u|^2-\int_{\mathcal{M}_h}\psi' \Lambda^{1,ij}  \overline{D_i u} D_j u.
$$
Taking $\psi'$ affine (and chosing it not constant), the first part controls the $L^2$ norm of $\partial_r u$. In the subflat and flat cases, we chose $\psi'$ linear because $\Lambda^{1,ij}$ is proportional to $-\frac1{r} h^{ij}$ thus we need $r$ to compensate this loss. In the hyperbolic case, we have $\Lambda^{1,ij}$ proportional to $-h^{ij}$ thus we need a constant term. 

Above $R$, we take $\psi'(r) = A -Bd(r)^{-1}$, choosing $A$ and $B$ such that $\psi$ is $\mathcal C^2$. This has many advantages : since $d$ is increasing, $\psi'$ is increasing and positive; taking the Laplace-Beltrami of $\psi$ yields $\lap_h \psi = A \lap_h r = A\frac{d'}{d}$; this choice makes $\psi''$ differentiable but not $\mathcal C^1$, which induces a Dirac delta in $\lap_h^2 \psi$.

\subsection{Hyperbolic-type metrics}
We start with the choice $d(r)=e^{r/2}$ that, as one may re-scale, includes some hyperbolic manifolds. In this case we have
$$
\mathcal R_h = -\frac38 + e^{-r/2} \mathcal R_\kappa \, ,\, \partial_1 \mathcal R_h = -\frac12e^{-r/2} \mathcal R_\kappa \, ,\, \Lambda^{1,ij} = -\frac14 h^{ij}.
$$

We recall that under the hypothesis of Theorem \ref{locsmoothteo2} for the hyperbolic type metrics, the curvature of $\kappa$ is positive, we recall our notation : $\mathcal R_\kappa > 0$.

We make the following choice for the radial multiplier $\psi_R$:
\begin{equation}\label{defpsiWP}
\psi'_R(r) = \left \lbrace{ \begin{array}{ll}
\displaystyle 1+M re^{- R/2} & \textrm{ if } r\leq R \\
\displaystyle 1+M R e^{-R/2 }(2+R)-2M e^{-r/2} & \textrm{ if } r> R.
\end{array}} \right.
\end{equation}
 for some $M \leq \mathcal \inf \mathcal{R}_\kappa$ where, we recall, $\psi_R' = \partial_1\psi_R=\partial_r\psi_R$.
With this choice, we have that $\psi$ is $\mathcal C^2$ and the following identities hold
\begin{eqnarray*}
\psi_R'' & =&   M e^{-R/2}{\bf 1}_{r\leq R} +M e^{-r/2} {\bf 1}_{r>R}\\
\lap_h \psi_R &= &\Big( \frac12 + M e^{-R/2}(1+\frac{r}{2})\Big) {\bf 1}_{r\leq R} + {\bf 1}_{r>R} \Big(\frac12 +M  e^{-R/2}(2+R)\Big)\\
\lap^2 \psi_R &= &- \frac{M}{2}e^{-R/2} \delta(r-R) + {\bf 1}_{r\leq R} \frac{M}{4}e^{-R/2}
\end{eqnarray*}
where $\delta$ is the Dirac delta. 

We start by computing the terms involving $|u|^2$ in the virial identity.

\begin{lemma}\label{lem-hypcompu2} We have
$$
-\frac12 \int_{\mathcal{M}_h} \Big( \lap_h^2 \psi + \frac12\grad_h \psi \cdot \grad_h \mathcal R_h\Big) |u|^2 \geq \int_{S_R} d\kappa |u|^2
$$
where $S_R$ is the set $r=R$. \end{lemma}

\begin{proof} We use the fact that $\mathcal R_\kappa $ is positive, that $\psi' \geq 0$ and that below $R$, $\psi'(r) \geq 1$, to get  
$$
-{\bf 1}_{r\leq R} \frac{M}{4}e^{-R/2} - \frac12 \mathcal \grad_h \psi \cdot \grad_h \mathcal R_h \geq {\bf 1}_{r\leq R} \Big( e^{-r/2} \frac14 \mathcal R_\kappa - e^{-R/2} \frac{M}{4}\Big).
$$
Since $M \leq \inf \mathcal R_\kappa$, we get
$$
-{\bf 1}_{r\leq R} \frac{M}{4}e^{-R/2} - \frac12 \mathcal \grad_h \psi \cdot \grad_h \mathcal R_h \geq 0.
$$

Therefore, we have 
$$
-\frac12 \int_{\mathcal{M}_h} \Big( \lap_h^2 \psi + \frac12\grad_h \psi \cdot \grad_h \mathcal R_h\Big) |u|^2 \geq \int_{\mathcal{M}_h} \frac{M}{2}e^{-R/2} \delta(r-R) |u|^2.
$$
What is more,
$$
\int_{\mathcal{M}_h} \delta(r-R) |u|^2 =  e^{R/2} \int_{S_R} |u|^2d\kappa
$$
which yields the result.
\end{proof}

We now deal with the terms involving the gradient of $u$. 

\begin{lemma}\label{lem-hypcompgrad} Assuming $M \leq \frac14$, we have
$$
\int_{\mathcal{M}_h} D^2(\psi)^{ij} \overline{D_i u} D_j u \geq M e^{-R/2} \int_{B_R} |\grad_h u|^2.
$$
\end{lemma}

\begin{proof} If $i=j=0$, we have $D^2(\psi)^{ij} = \psi''(r)$. And if none of them is $0$, we have $D^2(\psi)^{ij} = \frac14 h^{ij}\psi'$. We use that above $R$, $\psi'$ and $\psi''$ are non negative to get
$$
\int_{r> R} D^2(\psi)^{ij} \partial_i \overline u \partial_j u \geq 0.
$$
Below $R$, we use that $\psi'' \geq M e^{-R/2} $ and $\psi' \geq 1 \geq M e^{-R/2} $, to get the result.
\end{proof}

\begin{lemma}\label{lem-hypnormu2} We have, for any $\eta > 0$
$$
\sup_R \int_{S_R} \int dt \phi^{-1}(t)|u^2| d\kappa \geq \int_{\mathcal{M}_h}\an{r}^{-(1+\eta)} e^{-r/2} \int dt \phi^{-1}(t)|u|^2 .
$$
\end{lemma}

\begin{proof}
Indeed,
$$
\int_{\mathcal{M}_h}\an{r}^{-(1+\eta)} e^{-r/2} \int dt \phi^{-1}(t)|u|^2 = \int_{0}^\infty \an{r}^{-(1+\eta)} \int_{S_r} \int dt \phi^{-1}(t) |u|^2d\kappa.
$$
\end{proof}

\begin{lemma}\label{lem-hypnormgrad} For any $\eta_2 > 0$, there exists $C_{\eta_2}$ such that
$$
C_{\eta_2} \sup_R e^{-R/2} \int_{B_R} \int dt \phi^{-1}(t) |\grad_h u|^2 \geq \int_{\mathcal{M}_h} \an{r}^{-(1+\eta_2)} e^{-r/2} \int dt \phi^{-1}(t) |\grad_h u|^2.
$$
\end{lemma}

\begin{proof}
Indeed, let $\chi = \an{r}^{-(1+\eta_2)} e^{-r/2}$, we have
$$
\int_{\mathcal{M}_h} \chi(r)  \int dt \phi^{-1}(t)|\grad_h u|^2 = -\int_{\mathcal{M}_h} \int_{r}^\infty \chi'(y) dy  \int dt \phi^{-1}(t)|\grad_h u|^2(r).
$$
Interverting the integrals we get
\begin{multline*}
\int_{\mathcal{M}_h} \chi(r)  \int dt \phi^{-1}(t)|\grad_h u|^2 =- \int_{0}^\infty dy \chi'(y) \int_{B_y} \int dt \phi^{-1}(t)|\grad_h u|^2
\\
 \leq -\int_{0}^\infty \chi'(y) e^{y/2} dy \sup_R e^{-R/2}\int_{B_R} \int dt \phi^{-1}(t) |\grad_h u|^2.
\end{multline*}
We have $\chi'(y)e^{y/2} = -(1+\eta_2) \frac{r}{\an{r}^{3+\eta}} - \frac12 \an{r}^{-(1+\eta_2)}$. Hence it is integrable, and we get the result.
\end{proof}

\begin{lemma}\label{lem-RHS}Under the hypothesis of Theorem \ref{locsmoothteo2}, there exists $C$ such that for every $u$ solution of the linear Dirac equation, we have the following estimate 
$$
\Big| \int_{\mathcal{M}_h} \lap \psi_R \overline u \phi \partial_t u + \int_{\mathcal{M}_h} \grad_h \psi_R \cdot \grad_h \overline u \phi \partial_t u \Big| \leq C \|Hu(t)\|_{L^2(\mathcal{M}_h)}.
$$
\end{lemma}

\begin{proof} We use that $u$ is a solution to the Dirac equation, that is $\phi \partial_t u = Hu$,
$$\Big| \int_{\mathcal{M}_h} \lap \psi_R \overline u \phi \partial_t u + \int_{\mathcal{M}_h} \grad_h \psi_R \cdot \grad_h \overline u \phi \partial_t u \Big| = \Big| \int_{\mathcal{M}_h} \lap_h \psi_R \overline u Hu + \int_{\mathcal{M}_h} \grad_h \psi_R \cdot \grad_h \overline u H u \Big| .
$$
Then, we use that $\lap_h \psi_R$ and $\grad_h \psi_R $ both belong to $L^\infty$ and that their $L^\infty$ norms are uniformly bounded in $R$, to obtain
$$
\Big| \int_{\mathcal{M}_h} \lap \psi_R \overline u \phi \partial_t u + \int_{\mathcal{M}_h} \grad_h \psi_R \cdot \grad_h \overline u \phi \partial_t u \Big| \lesssim \int_{\mathcal{M}_h}  |\overline u| \; | Hu| +\int_{\mathcal{M}_h} | \grad_h \overline u|\; | H u|  .
$$
Thanks to H\"older's inequality we get
$$
\Big| \int_{\mathcal{M}_h} \lap \psi_R \overline u \phi \partial_t u + \int_{\mathcal{M}_h} \grad_h \psi_R \cdot \grad_h \overline u \phi \partial_t u \Big| \lesssim \| u\|_{L^2} \| Hu\|_{L^2} +\| \grad_h u\|_{L^2} \| H u\|_{L^2}  .
$$
We use that under the hypothesis of Theorem \ref{locsmoothteo2}, $\mathcal R_h + 4m^2 $ is more than a non-negative constant, which explains the hypothesis $m^2 > \frac3{32}$, to get that 
$$
\|u\|_{L^2} + \|\grad_h u\|_{L^2} \lesssim \|H u\|_{L^2}
$$
which yields the result. \end{proof}

\begin{proof}[Proof of estimate \eqref{locsmoothhyp} in Theorem \ref{locsmoothteo2}. ] Let us recall the virial identity
\begin{multline*}
-\int_{\mathcal{M}_h}\Big( \frac12 \lap_h^2 \psi_R + \frac14 \grad_h \psi_R \cdot \grad_h \mathcal R_h\Big) |u|^2 + 2 \int_{\mathcal{M}_h} (\partial_j \overline u \partial_i u) D^2(\psi_R)^{ij}\\
=-\phi \partial_t \re \Big( \int_{\mathcal{M}_h} \lap_h \psi_R \overline u \partial_t u + 2 \int_{\mathcal{M}_h} \grad_h \psi_R \cdot \grad_h \overline u \phi \partial_t u \Big).
\end{multline*}

We divide by $\phi$ and integrate over time to get
\begin{multline*}
\int dt \phi^{-1}(t) \Big(-\int_{\mathcal{M}_h}\Big( \frac12 \lap_h^2 \psi_R + \frac14 \grad_h \psi_R \cdot \grad_h \mathcal R_h\Big) |u|^2 + 2 \int_{\mathcal{M}_h} (\partial_j \overline u \partial_i u) D^2(\psi_R)^{ij}\Big)\\
\leq 2 \sup_t \Big| \re \Big( \int_{\mathcal{M}_h} \lap_h \psi_R \overline u \partial_t u + 2 \int_{\mathcal{M}_h} \grad_h \psi_R \cdot \grad_h \overline u \phi \partial_t u \Big)\Big|.
\end{multline*}

We use Lemma \ref{lem-RHS} to get
\begin{multline*}
\int dt \phi^{-1}(t) \Big(-\int_{\mathcal{M}_h}\Big( \frac12 \lap_h^2 \psi_R + \frac14 \grad_h \psi_R \cdot \grad_h \mathcal R_h\Big) |u|^2 + 2 \int_{\mathcal{M}_h} (\partial_j \overline u \partial_i u) D^2(\psi_R)^{ij}\Big)\\
\leq C \sup_t \|Hu(t)\|_{L^2}^2.
\end{multline*}

And finally, we use the conservation of energy to get
\begin{multline*}
\int dt \phi^{-1}(t) \Big(-\int_{\mathcal{M}_h}\Big( \frac12 \lap_h^2 \psi_R + \frac14 \grad_h \psi_R \cdot \grad_h \mathcal R_h\Big) |u|^2 + 2 \int_{\mathcal{M}_h} (\partial_j \overline u \partial_i u) D^2(\psi_R)^{ij}\Big)\\
\leq C \sup_t \|Hu_0\|_{L^2}^2.
\end{multline*}

We use Lemmas \ref{lem-hypcompu2} and \ref{lem-hypcompgrad} :
$$
\int dt \phi^{-1}(t) \Big(\int_{S_R} d\kappa |u|^2 + M e^{-R/2} \int_{B_R} |\grad_h u|^2 \Big)
\leq C \sup_t \|Hu_0\|_{L^2}^2,
$$
that is,
$$
 \Big(\int_{S_R} d\kappa \int dt \phi^{-1}(t) |u|^2 + M e^{-R/2} \int_{B_R}\int dt \phi^{-1}(t) |\grad_h u|^2 \Big)
\leq C \sup_t \|Hu_0\|_{L^2}^2.
$$

Finally, we pass to the sup and use Lemmas \ref{lem-hypnormu2} and \ref{lem-hypnormgrad} to get
\begin{multline*}
 \Big(\int_{\mathcal{M}_h} \an{r}^{-1-\eta_1}e^{-r/2} \int dt \phi^{-1}(t) |u|^2 +   \int_{\mathcal{M}_h}\an{r}^{-1-\eta_2}e^{-r/2}\int dt \phi^{-1}(t) |\grad_h u|^2 \Big)\\
\leq C \sup_t \|Hu_0\|_{L^2}^2,
\end{multline*}
which gives the result.
\end{proof}

\subsection{Flat-type metrics}

There is an equivalence between the Dirac equation on $\R^{1+3}$ and on the warped product $\R\times \R_+ \times S^2$ provided one choses a natural dreibein. As we mentionned, dreibein connect the structure of the tangent spaces, it is therefore natural to consider
$$
f_j^a = \partial_j y^a
$$
where we use $y^a$ as the coordinates in the Euclidean space. In this case, the equality $f_j^a\eta_{ab} f_i^b = h^{ij}$ is equivalent to $dy^a \eta_{ab} dy^b = dx^i h_{ij} dx^j$. The spin connection provided with this dreibein is equal to $0$. An easy way to see this is to use a different (but equivalent) definition for the spin connection : 
$$
\alpha_i^{ab} = \frac12 f^{ja}( \partial_i f_j^b-  \partial_j f_i^b) - \frac12 f^{jb}(\partial_i f_j^a - \partial_j f_i^a) -\frac12 f^{ka}f^{jb} (\partial_k f_{jc } - \partial_j f_{kc}) f_i^c.
$$
Because the change of variable is smooth, we have
$$
\partial_i f_j^b = \partial_i \partial_j y^b = \partial_j f_i^b.
$$
We also have $f_{jc} = \eta_{cd}f_j^d$ therefore, 
$$
(\partial_k f_{jc } - \partial_j f_{kc}) = \eta_{cd} (\partial_k f_{j}^d - \partial_j f_{k}^d) = 0
$$
which yields $ \alpha_i^{ab} =0$. This gives $D_i = \partial_i$. Finally, the Dirac equation in this space and with this dreibein writes
$$
i\gamma^a e^\mu_a D_\mu u = mu.
$$
The mass term and the term involving the derivative in time do not change, hence we may focus on
$$
\gamma^a f^j_a D_j = \gamma^a f^j_a \partial_j 
$$
and since $\partial_j = \partial_j y^b \partial_b = f_j^b \partial_b$ we get
$$
\gamma^a f^j_a D_j = \gamma^a f^j_af_j^b \partial_b = \gamma^a \delta_a^b \partial_b = \gamma^a \partial_a
$$
and we retrieve the Dirac equation in the flat case.

We now take $d(r) = r^2$ which includes the flat case. With this choice we have 
$$
\mathcal R_h  = -2 r^{-2} + r^{-2} \mathcal R_\kappa , \partial_r \mathcal R_h = - 2(\mathcal R_\kappa - 2) \frac1{r^3}.
$$
If $\mathcal R_\kappa \geq 2$, then the computations are exactly the same as in the flat case. Let us be more precise. 

The multiplier should be essentially the same as in the flat case, that is :  
$$
\psi'_R(r) = \frac{r}{\an R} {\bf 1}_{r\leq R} + \frac{R}{\an R} \Big(  \frac32 - \frac{R^2}{r^2}\Big) {\bf 1}_{r>R}.
$$
We multiply the usual multiplier by $\frac{R}{\an R}$ not to mess with Hardy's inequality in the energy term.

We recall that $\psi'_R$ is non-negative, increasing and bounded by $\frac32$.

What is more,
$$
\lap_h \psi_R = \frac{3}{\an R} {\bf 1}_{r\leq R} + \frac{R}{\an R} \frac3{r} {\bf 1}_{r>R}
$$
which is positive and bounded by $3$ and 
$$
\lap_h^2 \psi_R = -\frac3{R \an R} \delta (r-R).
$$

We start with the terms involving $|u|^2$. 

\begin{lemma} For all $R$, we have
$$
 -\frac12 \int_{\mathcal{M}_h} (\lap_h^2 \psi_R + \frac12 \grad_h \psi_R\cdot \grad_h \mathcal R_h ) |u|^2 \geq  3 \frac{R}{\an R} \int_{S_R}|u|^2 d\kappa.
 $$
And for all $\eta > 0$, there exists $C_\eta$ such that
$$
\|\an{r}^{-3/2-\eta}u\|_{L^2(\mathcal{M}_g)}^2 \leq C_\eta \sup_R \frac{R}{\an R} \int_{S_R}\int dt \phi^{-1}(t) |u|^2d\kappa.
$$
\end{lemma}

\begin{proof} We have that
$$
-\int_{\mathcal{M}_h} \lap_h^2 \psi_R |u|^2 =  3 \frac{R}{\an R} \int_{S_R} |u|^2 d\kappa.
$$
What is more $\psi'_R$ is non-negative and since $\mathcal R_\kappa \geq 2$, $\partial_r \mathcal  R_h$ is non positive. This means 
$$
-\frac14 \int_{\mathcal{M}_h}  \grad_h \psi_R\cdot \grad_h \mathcal R_h  |u|^2\geq 0.
$$
Note that since $\psi'_R$ goes to $\frac32  \frac{R}{\an R} $ when $r$ goes to $\infty$, $ \grad_h \psi_R\cdot \grad_h \mathcal R_h $ behave like $\frac1{r^3}$ when $r$ goes to $\infty$, this is not sufficient to be compensated by the bi-Laplace-Beltrami term if $\mathcal R_\kappa \leq 2$.

We consider the quantity
\begin{multline*}
\|\an{r}^{-3/2-\eta}u\|_{L^2(\mathcal{M}_g)}^2 = \int_{\mathcal{M}_h} \an{r}^{-3-2\eta}\int dt \phi^{-1}(t) |u|^2 \\
= \int_{0}^\infty \an{r}^{-2-2\eta} r\Big( \frac{r}{\an r} \int_{S_r} \int dt \phi^{-1}(t)|u|^2 d\kappa\Big).
\end{multline*}
Since $\an{r}^{-2-2\eta} r$ is integrable, this yields,
$$
\|\an{r}^{-3/2-\eta}u\|_{L^2(\mathcal{M}_g)}^2 \leq C_\eta \sup_R \frac{R}{\an R} \int_{S_R} \int dt \phi^{-1}(t)|u|^2 d\kappa
$$
which concludes the proof.
\end{proof}

We now focus on the term involving the derivatives of $u$.

\begin{lemma} For all $R$, we have 
$$
\int_{\mathcal{M}_h} D^2(\psi_R)^{ij} \overline{D_i u} D_ju  \geq \frac1{\an R} \int_{B_R} |\grad_h u|^2.
$$
And for any $\eta_2 > 0$, there exists $C_{\eta_2}$ such that
$$
\|\an{r}^{-1/2-\eta_2} \grad_h u\|_{L^2(\mathcal{M}_g)} \leq C_{\eta_2} \sup_R \frac1{\an R} \int_{B_R} \int dt \phi^{-1}(t) |\grad_h u|^2.
$$
\end{lemma}

\begin{proof} We have 
$
D^2(\psi_R)^{ij} = h^{il}h^{kj}\partial_l \partial_k \psi_R - \Lambda^{k,ij} \partial_k \psi_R,
$
and because $\psi_R$ is radial, this yields
$$
D^2(\psi_R)^{ij} = \delta^i_1 \delta^j_1 \psi_R'' - \Lambda^{1,ij} \psi_R'.
$$
we have $\Lambda^{1,ij} =0$ if $i=1$ or $j=1$ and $\Lambda^{1,ij} = - \frac12 \frac{d'}{d} h^{ij}$ otherwise. Therefore, with $d(r) = r^2$
$$
D^2(\psi_R)^{ij} = \delta^i_1 \delta^j_1 \psi_R'' + (1-\delta^i_1)(1- \delta^j_1)\frac{h^{ij}}{r} \psi_R'.
$$
Above $R$, we use that $\psi'_R$ and $\psi''_R$ are non negative and that $\kappa^{ij}$ is positive to get
$$
\int_{r>R} D^2(\psi_R)^{ij}  \overline{D_i u} D_j u \geq 0.
$$

Below $R$, we have $\psi''_R = \frac{\psi'_R}{r} = \frac1{\an R}$, therefore $D^2(\psi_R)^{ij} =\frac1{\an R} h^{ij}$ and
$$
\int_{\mathcal{M}_h}D^2(\psi_R)^{ij} \partial_i \overline u \partial_j u \geq \frac1{\an R}\int_{B_R} |\grad_h u|^2 .
$$

Write $\chi(r) = \an{r}^{-1-2\eta_2}$. We use as in the hyperbolic case that
$$
\int_{\mathcal{M}_h} \chi(r) \int dt \phi^{-1}(t)|\grad_h u|^2  = - \int_{0}^\infty \chi'(y) \int_{B_y} \int dt \phi^{-1}(t)|\grad_h u|^2
$$
which yields
$$
\int_{\mathcal{M}_h} \chi(r) \int dt \phi^{-1}(t)|\grad_h u|^2  = - \int_{0}^\infty \chi'(y)\an y \frac1{\an y} \int_{B_y} \int dt \phi^{-1}(t) |\grad_h u|^2
$$
We have $-\chi'(y) \an y = (1+2\eta_2) y \an{y}^{-2-2\eta_2}$ which is integrable since $\eta_2 >0$ and thus
$$
\int \chi(r)  \int dt \phi^{-1}(t)|\grad_h u|^2  \lesssim \sup_y \frac1{\an y} \int_{B_y} \int dt \phi^{-1}(t) |\grad_h u|^2
$$
from which we deduce the result.
\end{proof}

This concludes the estimates for the LHS, we now deal with the RHS.

\begin{lemma}Under the hypothesis of Theorem \ref{locsmoothteo2}, there exists $C$ such that for every $u$ solution of the linear Dirac equation, we have the following estimate 
$$
\Big| \int_{\mathcal{M}_h} \lap \psi_R \overline u \phi \partial_t u + \int_{\mathcal{M}_h} \grad_h \psi_R \cdot \grad_h \overline u \phi \partial_t u \Big| \leq C \|Hu(t)\|_{L^2(\mathcal{M}_h)}.
$$
\end{lemma}

\begin{proof} We can essentially repeat the proof in the hyperbolic case. The only difference is that to have $\mathcal R_h + 4m$ bigger than a positive constant, we need $m >0$. \end{proof}

The second estimate of Theorem \ref{locsmoothteo2} is deduced in the same way as in the hyperbolic case.

\subsection{Sub-flat type metrics}

We consider another specific case, which is $d(r) = r^n$ with $2-\sqrt 2 < n \leq \frac43$. With this choice we have
\begin{equation}\label{Rint}
\mathcal R_h = \frac{4n-3n^2}{2r^2} + \frac1{r^n}\mathcal R_\kappa \mbox{ and } \partial_r \mathcal R_h = -\frac{4n-3n^2}{r^3} - \frac{n}{r^{n+1}} \mathcal R_\kappa.
\end{equation}
As $\partial_r \mathcal R_h $ is negative (as long as $\mathcal R_\kappa$ is non negative), we may have that it can compensate losses due to the bi-Laplacian. 

Let us take
$$
\psi'_R = \left \lbrace{ \begin{array}{cc}
\displaystyle\frac{r}{\an{R}} & \textrm{ if } r\leq R \\
\displaystyle\Big( \frac{n+1}{n} -\frac1{n} \Big( \frac{R}{r}\Big)^n\Big) \frac{R}{\an{R}} & \textrm{ if } r> R
\end{array}}\right.
$$
Notice that with this choice  $\psi_R\in \mathcal C^2$. Moreover, we note the following properties :
\begin{itemize}
\item
for $r\leq R$, we have $\psi''_R = \frac1{\an{R}}$, and $\lap_h \psi_R = \frac{n+1}{\an{R}}$,
\item
For $r> R$, we have $\psi''_R > 0$, hence $\frac{R}{\an{R}}\leq  \psi'_R  \leq \frac{(n+1) R}{\an{R}}$. 
\end{itemize}
Besides, 
$$\lap_h \psi_R = (n+1) \frac{R}{r\an{R}}{\bf 1}_{r>R}.$$

From these relations we can deduce, 
$$
\lap_h^2 \psi_R = - \frac{(n+1)}{R\an{R}} \delta (r-R) - {\bf 1}_{r> R} R(n+1)(n-2) r^{-3}\an{R}^{-1}.
$$

Note that $\psi_R'$ is non negative, increasing and bounded by $\frac{n+1}{n}$ and that $\lap_h \psi_R$ is bounded by $n+1$. 

We start with the terms in $|u|^2$ in the LHS.

\begin{lemma} Assuming that $n\geq 2-\sqrt 2$, there exists $C_n >0$, such that if $\min \mathcal R_\kappa \geq C_n$, then 
\begin{multline*}
\sup_y \frac{y^{n-1}}{\an y} \int_{S_y} \int dt \phi^{-1}(t)|u|^2 d\kappa \leq\\
 C_n \sup_R \Big(-\int_{\mathcal{M}_h}( \lap_h^2\psi_R + \frac12 \grad_h \mathcal R_h \cdot \grad_h \psi_R)  \int dt \phi^{-1}(t)|u|^2 \Big) .
\end{multline*}
\end{lemma}

\begin{proof} Write 
$$
-(\lap_h^2\psi_R + \frac12 \grad_h \mathcal R_h \cdot \grad_h \psi_R) = \frac{n+1}{R\an R} \delta(r-R) + f_R(r) +g_R(r)
$$
with
$$
f_R(r) = -{\bf 1}_{r\leq R} \frac12 \grad_h \mathcal R_h \cdot \grad_h \psi_R
$$
and 
$$
g_R(r) = \Big( \frac{R}{\an R}(n+1)(n-2) r^{-3} - \frac12 \grad_h \mathcal R_h \cdot \grad_h \psi_R\Big){\bf 1}_{r> R}.
$$

We have 
$$
\int_{\mathcal{M}_h} \frac{n+1}{R\an R} \delta(r-R) \int dt \phi^{-1}(t)|u|^2  = (n+1) \frac{R^{n-1}}{\an R} \int_{S_R}  \int dt \phi^{-1}(t)|u|^2 d\kappa.
$$

We have that $\psi'_R$ is non negative and that $\partial_1 \mathcal R_h$ is non positive, thus $f_R(r) \geq 0$, therefore
$$
\int_{\mathcal{M}_h} f_R(r)  \int dt \phi^{-1}(t)|u|^2 \geq 0.
$$

Since $\partial_1 \mathcal R_h$ is non negative, and $\psi'_R \geq \frac{R}{\an R} $ we have 
$$
g_R(r) \geq \frac{R}{\an R } \Big( (n+1)(n-2) r^{-3} - \frac12 \partial_1 \mathcal R_h\Big){\bf 1}_{r> R},
$$
and replacing $\partial_1 \mathcal R_h$ by its value,
$$
g_R(r) \geq  \frac{R}{\an R } \Big( -\frac{n^2-2n +4}{2} r^{-3} +\frac{n\mathcal R_\kappa}{r^{n+1}}\Big){\bf 1}_{r> R}.
$$
We factorise by $r^{-3}$ to get
$$
g_R(r) \geq  \frac{R}{r^3\an R } \Big( -\frac{n^2-2n +4}{2}  +n\mathcal R_\kappa r^{2-n+1}\Big){\bf 1}_{r> R}.
$$
Note that $n^2 - 2n + 4$ is always positive and that $2-n>0$. Hence for $r\geq R_0 = \Big( \frac{n^2-2n +4}{2n\min \mathcal R_\kappa}\Big)^{1/(2-n)}$, $g_R(r) \geq 0$, which implies
$$
\int_{\mathcal{M}_h} g_R(r) |u|^2 \geq \int_{R\leq r\leq R_0} g_R(r) |u|^2 dr \geq - \frac{R}{\an R }\frac{n^2-2n +4}{2} \int_{R\leq r\leq R_0} r^{-3}|u|^2.
$$
We have
$$
\int_{R\leq r\leq R_0} r^{-3} \int dt \phi^{-1}(t)|u|^2 = \int dr r^{n-3} \int_{S_r } \int dt \phi^{-1}(t)|u|^2d\kappa
$$
$$
\leq \sup_y \frac{y^{n-1}}{\an y} \int_{S_y}  \int dt \phi^{-1}(t)|u|^2 d\kappa \int_{R}^{R_0} \an r r^{-2}dr.
$$
Since
$$
\int_{R}^{R_0} \an r r^{-2}dr \leq \frac{\an{R_0}}{R},
$$
we have
$$
\int_{\mathcal{M}_h} g_R(r) \int dt \phi^{-1}(t) |u|^2 \geq - \frac{n^2-2n +4}{2} \an{R_0} \sup_y \frac{y^{n-1}}{\an y}\int_{S_y} \int dt \phi^{-1}(t)|u|^2.
$$
Combining this inequality with the ones involving $f_R$ and the Dirac delta, we get
$$
-\int_{\mathcal{M}_h}( \lap_h^2\psi_R + \frac12 \grad_h \mathcal R_h \cdot \grad_h \psi_R) \int dt \phi^{-1}(t) |u|^2
$$
$$
 \geq (n+1) \frac{R^{n-1}}{\an R} \int_{S_R} \int dt \phi^{-1}(t) |u|^2 d\kappa -  \frac{n^2-2n +4}{2} \an{R_0} \sup_y \frac{y^{n-1}}{\an y}\int_{S_y} \int dt \phi^{-1}(t) |u|^2.
$$
we take the supremum over $R$ to get
$$
\sup_R \Big(-\int_{\mathcal{M}_h}( \lap_h^2\psi_R + \frac12 \grad_h \mathcal R_h \cdot \grad_h \psi_R) \int dt \phi^{-1}(t) |u|^2\Big)
$$
$$
 \geq \Big( (n+1)  -  \frac{n^2-2n +4}{2} \an{R_0}\Big) \sup_y \frac{y^{n-1}}{\an y}\int_{S_y} \int dt \phi^{-1}(t)|u|^2.
$$
Note that $n+1- \frac{n^2-2n +4}{2} = \frac12 (2+\sqrt2 -n)(n-(2-\sqrt2)$ hence if $n> 2-\sqrt 2$, $\frac{2(n+1)}{n^2-2n +4} > 1$. Take $\inf \mathcal R_\kappa$ sufficiently big such that $R_0$ is sufficiently small to have 
$$
\an{R_0} < \frac{2(n+1)}{n^2-2n +4}.
$$
With these conditions, $(n+1)  -  \frac{n^2-2n +4}{2} \an{R_0} > 0$, and we get the result.
\end{proof}

\begin{lemma} Let $n \in ]2-\sqrt 2, \frac43 ]$ and $\eta >0$. There exists $C_{n,\eta}$ such that if for all $x^2,x^3$, $\mathcal R_\kappa \geq C_n$ then
$$
\|\an{r}^{-3/2-\eta}u\|_{L^2(\mathcal{M}_g)}^2 \leq C_n
\sup_R \Big(-\int_{\mathcal{M}_h}( \lap_h^2\psi_R + \frac12 \grad_h \mathcal R_h \cdot \grad_h \psi_R)  \int dt \phi^{-1}(t)|u|^2\Big).
$$
\end{lemma}

\begin{proof} We have
\begin{multline*}
\|\an{r}^{-3/2-\eta}u\|_{L^2(\mathcal{M}_g)}^2 = \int_{\mathcal{M}_h} \an{r}^{-3-2\eta} \int dt \phi^{-1}(t) |u|^2\\
 = \int_{0}^\infty dr r^n \an{r}^{-3-2\eta} \int_{S_r}d\kappa \int dt \phi^{-1}(t) |u|^2.
\end{multline*}

We get
$$
\|\an{r}^{-3/2-\eta}u\|_{L^2(\mathcal{M}_g)}^2 \leq \int_{0}^\infty dr \an r^{-2-2\eta} r \sup_{y} \frac{y^{n-1}}{\an y} \int_{S_y}  \int dt \phi^{-1}(t)|u|^2 d\kappa
$$
and we use the previous result to conclude.
\end{proof}

We deal with the terms in $\grad_h u$ in the left hand side of the virial identity.

\begin{lemma} We have for all $R$,
$$
2 \int_{\mathcal{M}_h} D^2(\psi_R)^{ij} \overline{D_i u} D_j u \geq \frac{n}{\an R}\int_{B_R} |u|^2 .
$$
For all $\eta_2 >0$ there exists $C_{\eta_2, n}$ such that
$$
\|\an{r}^{-1/2-\eta_2}\grad_h u\|_{L^2(\mathcal{M}_g)}^2 \leq C_{\eta_2,n} \sup_R \frac1{\an R} \int_{B_R} \int dt \phi^{-1}(t)|\grad_h u|^2.
$$
\end{lemma}

\begin{proof}
We have 
$$
D^2(\psi_R)^{11} = \psi''_R \mbox{ and if }ij\neq 1\;  D^2(\psi_R)^{ij} = \frac12 \frac{n}{r} h^{ij} (\psi'_R),
$$
from which we deduce that above $R$, since $\psi_R'$ is non negative and non decreasing, we have 
$$
D^2(\psi_R)^{ij} \partial_i \overline{D_i u} D_j u \geq 0
$$
and under $R$, since $\psi'_R = \frac{r}{\an R}$,
$$
D^2(\psi_R)^{ij}\overline{D_i u} D_j u \geq  \frac{n}{2\an{R} } |\grad_h u|^2 .
$$
Therefore,
$$
2 \int_{\mathcal{M}_h} D^2(\psi_R)  \partial_i \overline u \partial_j u  \geq \frac{n}{\an{R}} \int_{B_R} |\grad_h u|^2.
$$

Let $\chi(r) = \an r^{-1-2\eta_2}$.
We have
$$
\|\an{r}^{-1/2-\eta_2}\grad_h u\|_{L^2(\mathcal{M}_g)}^2 = \int_{\mathcal{M}_h} \chi(r) \int dt \phi^{-1}(t) |\grad_h u|^2.
$$
Given that $\chi(r) = -\int_{r}^\infty \chi'(y)dy$, we get
$$
\|\an{r}^{-1/2-\eta_2}\grad_h u\|_{L^2(\mathcal{M}_g)}^2 = -\int_{0}^\infty dy \chi'(y) \int_{B_y}  \int dt \phi^{-1}(t)|\grad_h u|^2.
$$
We deduce
$$
\|\an{r}^{-1/2-\eta_2}\grad_h u\|_{L^2(\mathcal{M}_g)}^2 \leq -\int_{0}^\infty \an y \chi'(y) \sup_R \frac1{\an R} \int_{B_R} \int dt \phi^{-1}(t) |\grad_h u|^2.
$$
We have $-\chi'(y) \an y=(1+2\eta_2) \frac{r}{\an r} \an{r}^{-1-2\eta_2}$, hence $-\chi'(y)\an y$ is integrable and we get the result.
\end{proof}

We now deal with the LHS.

\begin{lemma}Under the hypothesis of Theorem \ref{locsmoothteo2}, there exists $C$ such that for every $u$ solution of the linear Dirac equation, we have the following estimate 
$$
\Big| \int_{\mathcal{M}_h} \lap \psi_R \overline u \phi \partial_t u + \int_{\mathcal{M}_h} \grad_h \psi_R \cdot \grad_h \overline u \phi \partial_t u \Big| \leq C \|Hu(t)\|_{L^2(\mathcal{M}_g)}.
$$
\end{lemma}

\begin{proof} We can essentially repeat the proof in the hyperbolic case. The only difference is that to have $\mathcal R_h + 4m$ bigger that a positive constant, we need $m >0$. \end{proof}

The third estimate of Theorem \ref{locsmoothteo2} is deduced in the same way as in the hyperbolic case.

{\bf Email Contacts:}
\begin{itemize}
\item Federico Cacciafesta: {\em cacciafe@math.unipd.it}
\item Anne-Sophie de Suzzoni: {\em adesuzzo@math.univ-paris13.fr}.
\end{itemize}


\begin{thebibliography}{9}

\bibitem{kg}
D. Baskin, \emph{A {S}trichartz estimate for de {S}itter space}, The
  {AMSI}-{ANU} {W}orkshop on {S}pectral {T}heory and {H}armonic {A}nalysis,
  Proc. Centre Math. Appl. Austral. Nat. Univ., vol.~44, Austral. Nat. Univ.,
  Canberra, 2010, pp. 97--104.   
\bibitem{wave00}
D. Baskin and Jared Wunsch, \emph{Resolvent estimates and local decay of
  waves on conic manifolds}, J. Differential Geom. \textbf{95} (2013), no.~2,
  183--214. 
  
\bibitem{wave0}
M. D. Blair, H. F. Smith and C. D. Sogge, \emph{Strichartz
  estimates for the wave equation on manifolds with boundary}, Ann. Inst. H.
  Poincar\'e Anal. Non Lin\'eaire \textbf{26} (2009), no.~5, 1817--1829.


\bibitem{boussdanfan}
N. Boussaid, P. D'Ancona and L. Fanelli,
\newblock Virial identity and weak dispersion for the magnetic dirac equation.
\newblock {\em Journal de Math\'{e}matiques Pures et Appliqu\'{e}es},
  95:137--150, 2011.

\bibitem{burq}
N. Burq,
\newblock Global Strichartz estimates for nontrapping geometries: about an article by H. F. Smith and C. D. Sogge. 
\newblock {\em Comm. Partial Differential Equations}, 28(9-10):1675-1683, 2003.

\bibitem{Cacciafesta2}
F. Cacciafesta,
\newblock Virial identity and dispersive estimates for the n-dimensional Dirac equation.
\newblock {\em J. Math. Sci. Univ. Tokyo} 18, 1-23, 2011.

\bibitem{cacsmooth}
F. Cacciafesta,
\newblock Smoothing estimates for variable coefficients Schroedinger equation with electromagnetic potential
\newblock {\em J. Math. Anal. Appl.} 402, pp. 286-296, 2013.

\bibitem{cacdan}
F. Cacciafesta and P. D'Ancona,
\newblock Endpoint estimates and global existence for the nonlinear Dirac equation with a potential.
\newblock  {\em J. Differential Equations} 254, pp. 2233-2260, 2013.

\bibitem{cacdanluc}
F. Cacciafesta, P. D'Ancona and R. Luc\'a,
\newblock Helmholtz and dispersive equations with variable coefficients on external domains.
\newblock {\em SIAM J. Math. Anal.} 48 (2016), no.3 1798-1832.

\bibitem{cacser}
F. Cacciafesta and Eric S\'er\'e.
\newblock Local smoothing estimates for the Dirac Coulomb equation in 2 and 3 dimensions.
\newblock {\em J. Funct. Anal.} 271  no.8, 2339-2358 (2016)

\bibitem{selfa}
Paul~R. Chernoff, \emph{Essential self-adjointness of powers of generators of
  hyperbolic equations}, J. Functional Analysis \textbf{12} (1973), 401--414.
 

 \bibitem{DanconaFanelli2}
P. D'Ancona and L. Fanelli,
\newblock Decay estimates for the wave and Dirac equations with a magnetic potential. 
\newblock {\em Comm. Pure Appl. Math.} 60,, no. 3, 357-392, 2007.

\bibitem{DanconaFanelli08-a}
P. D'Ancona and L. Fanelli,
\newblock Strichartz and smoothing estimates of dispersive equations with
  magnetic potentials.
\newblock {\em Comm. Partial Differential Equations}, 33(4-6):1082--1112, 2008.


\bibitem{FanelliVega09-a}
L. Fanelli and L. Vega,
\newblock Magnetic virial identities, weak dispersion and {S}trichartz
  inequalities.
\newblock {\em Math. Ann.}, 344(2):249--278, 2009.

\bibitem{fock}
V. Fock,
\newblock Geometrization of the Dirac thory of electrons.
\newblock {\em Zeit. f. Phys} 57, 261-277, 1929.

\bibitem{wave01}
A. Hassell, T. Tao and J. Wunsch, \emph{Sharp {S}trichartz
  estimates on nontrapping asymptotically conic manifolds}, American Journal of
  Mathematics \textbf{128} (2006), no.~4, 963--1024.

\bibitem{metctat}
J. Metcalfe and D. Tataru,
\newblock Global parametrices and dispersive estimates for variable coefficient wave equations. 
\newblock {\em Math. Ann.} 353 (2012), no. 4, 1183-1237.

\bibitem{parktoms}
L. E. Parker and D. J. Toms,
\newblock Quantum field theory in curved spacetime.
\newblock Cambridge university press.

\bibitem{riccilevi}
M. M. G. Ricci and T. Levi-Civita,
\newblock 
\newblock{\em Math. Ann.} 54, 125 (1900); 608 (E) (1901).


\bibitem{wave1}
H. F. Smith and C. D. Sogge, \emph{On the critical semilinear wave
  equation outside convex obstacles}, J. Amer. Math. Soc. \textbf{8} (1995),
  no.~4, 879--916. 

\bibitem{wave2}
D. Tataru, \emph{Strichartz estimates for second order hyperbolic operators
  with nonsmooth coefficients. {III}}, J. Amer. Math. Soc. \textbf{15} (2002),
  no.~2, 419--442 (electronic). 

\bibitem{wave3}
A. Vasy and J. Wunsch, \emph{Morawetz estimates for the wave
  equation at low frequency}, Math. Ann. \textbf{355} (2013), no.~4,
  1221--1254.
  
\end{thebibliography}
\end{document}